\documentclass[final, 10pt,reqno]{amsart}
\usepackage{amssymb,amsmath,graphicx,amsfonts,euscript}
\usepackage{color}
\usepackage[pagewise]{lineno}
\usepackage{showkeys}
\usepackage{stmaryrd}
\usepackage{fancyhdr}
\usepackage[margin=1in]{geometry}
\usepackage{titletoc}
\usepackage{hyperref}
\newtheorem{theorem}{\bf Theorem}[section]

\newtheorem{proposition}[theorem]{\bf Proposition}

\newtheorem{remark}{\bf Remark}

\newtheorem{lemma}[theorem]{\bf Lemma}

\newcommand{\beq}{\begin{equation}}
\newcommand{\eeq}{\end{equation}}
\newcommand{\ben}{\begin{eqnarray}}
\newcommand{\een}{\end{eqnarray}}
\newcommand{\beno}{\begin{eqnarray*}}
\newcommand{\eeno}{\end{eqnarray*}}

\numberwithin{equation}{section}
\makeatother
\title[Suppression of blow-up in 3-D  fractional Keller-Segel]{Suppression of blow-up in 3-D Keller-Segel system with fractional diffusion via Couette flow in whole space}

\author[S.Deng]{Shijin Deng}
\address{School of Mathematical Sciences, CMA-Shanghai, Shanghai Jiao Tong University, Shanghai 200240, P.R.China.}
\email{matdengs@sjtu.edu.cn}

\author[B.Shi]{Binbin Shi}
\address{School of Mathematics, Southwestern University of Finance and Economics,
Chengdu, 611130, P.R.China.}

\address{School of Mathematics and Statistics, Nanjing University of Science and Technology, Nanjing, 210094, P.R. China.}

\email{shibb@swufe.edu.cn}
\author[W.Wang]{Weike Wang}
\address{School of Mathematical Sciences, CMA-Shanghai and Institute of Natural Science, Shanghai Jiao Tong University, Shanghai, 200240, P.R.China.}
\email{wkwang@sjtu.edu.cn}

\author[Y.Wang]{Yucheng Wang}
\address{Department of Mathematics, Shanghai University, Shanghai 200444, China.}
\email{ycwangmath@shu.edu.cn}

\begin{document}

\begin{abstract}
In this paper, we consider a Keller-Segel model with a fractional diffusion term in $\mathbb{R}^3$ in the background of a Couette flow. We show that when the background Couette flow is large enough, the dissipation enhancement induced could prevent the blow-up of solutions and thus prove the global existence and also obtain time decay rates of the solution in $L^p$ norm. The main tool of the proof is a corresponding Green's function and the key estimate is its $L^1$ estimate without singularities at $t=0$. To fulfill such an estimate, we meet great troubles caused by the fractional heat kernel together with the Couette flow in the model considered here and overcome the troubles by introducing a space-frequency mixed decomposition.
\end{abstract}
\maketitle

\vspace{-1.2em}

\section{Introduction}

In this paper, we consider the following generalized Keller-Segel model in $\mathbb{R}^3$ in the background of a large Couette flow
\begin{equation}\label{eq:1.1}
\begin{cases}
\partial_tn+Ay\partial_xn+(-\Delta)^{\alpha/2} n+\nabla\cdot\big(n \mathbf{B}(n)\big)=0, \\
n(x,y,z,t)\big|_{t=0}=n_0(x,y,z),\ \ \ (x,y,z,t)\in\mathbb{R}^3\times\mathbb{R}^{+}.
\end{cases}
\end{equation}
Here, the non-negative unknown function ${n(x,y,z,t)}$ represents the density and $A$ is a positive constant. The fractional Laplacian  $(-\Delta)^{\alpha/2}$ is defined via the Fourier transform
\begin{equation}\label{eq:1.2}
(-\Delta)^{\alpha/2} n(x, y, z)=\int_{\mathbb{R}^3}(\xi^2+\eta^2+\zeta^2)^{\alpha/2}\widehat{n}(\xi,\eta,\zeta)e^{{i(\xi x+\eta y+\zeta z)}}d\xi d\eta d\zeta,\ \ \ 0<\alpha\leq2,
\end{equation}
where $\widehat{n}(\xi, \eta, \zeta)$ denotes the Fourier transform of $n(x, y, z)$. The linear vector operator $\mathbf{B}(n)$ is called the attractive kernel, which could be formally represented as
\begin{equation}\label{eq:1.3}
\mathbf{B}(n)=\nabla\big((-\Delta)^{-1}n\big).
\end{equation}
In this paper, we study the global well-posedness of equation \eqref{eq:1.1} with $A\gg 1$ in $\mathbb{R}^3$.

\vskip .05in

Without the advection $\big(i.e.\ A=0\ in\ \eqref{eq:1.1}\big)$, the equation \eqref{eq:1.1} becomes a generalized Keller-Segel model
\begin{equation}\label{eq:1.4}
\partial_tn+(-\Delta)^{\alpha/2} n+\nabla\cdot\big(n \mathbf{B}(n)\big)=0,
\end{equation}
which could be used to explain various biological and physical phenomena. When {{$\alpha=2$}}, the equation \eqref{eq:1.4} goes back to the classical parabolic-elliptic Keller-Segel model. It is well-known that the solutions in high dimensions may blow up in finite time if the initial function $n_0$ is large in $L^1$ norm.  More precisely, for two dimensional case, if $L^1$ norm of initial data {{$n_0$}} is less than $8\pi$, there exists a unique global solution; otherwise, the solution may blow up in finite time. One could refer to \cite{Nagai.1995} for more details. In three and higher dimensional cases, the solution may blow up even for an initial function $n_0$ arbitrarily small in $L^1$ norm (see \cite{SW.2019}). When $0<\alpha<2$, the solution of \eqref{eq:1.4} also may blow up in finite time for high dimensional cases under a large initial function $n_0$ (see \cite{BK.2010,LS.2019,LR.2009}).

\vskip .05in

The case $A\neq0$ is corresponding to a chemotactic process progressing in the background of a shear flow. A realistic scenario is that chemotactic processes take place in a non-stationary fluid, and the possible effects and related problems resulting from the interactions between the chemotactic process and the fluid transport have been widely investigated (see \cite{CLY.2015,DLM.2010,KLW.2022,LZ.2022,WW.2019}).  The study of the Keller-Segel model with the presence of an incompressible flow is one of those attempts:
\begin{equation}\label{eq:1.5}
\partial_tn+A{\bf u}\cdot\nabla n+(-\Delta)^{\alpha/2} n+\nabla\cdot\big(n \mathbf{B}(n)\big)=0,
\end{equation}
where ${\bf u}$ is a divergence-free vector field. An interesting question arising is whether the mixing effect coming from the fluid transport can suppress the possible finite time blow-up. Recently, some progresses have been made for the suppression of blow-up by incompressible flows. When $\alpha=2$, the equation \eqref{eq:1.5} is the classical Keller-Segel model with an incompressible flow. Kiselev-Xu \cite{KX.2016} and Hopf-Rodrigo \cite{HR.2018} considered the case that ${\bf u}$ is a relaxation enhancing flow (the definition of such a flow is given in \cite{CKRZ.2008}), and proved that the solution of the advective Keller-Segel equation does not blow up in finite time provided the amplitude of the relaxation enhancing flow ${\bf u}$ is large enough. Bedrossian-He \cite{BH.2017} found that shear flows \big(${\bf u}=(u(y),0)$\big) provide a different suppression effect in the sense that sufficiently large shear flows could prevent the blow-up in two dimensions but could not guarantee the global existence in three dimensions if the initial mass is greater than $8\pi$. When  $0<\alpha<2$, the equation \eqref{eq:1.5} is a generalized Keller-Segel model with an incompressible flow. Hopf-Rodrigo \cite{HR.2018} and Shi-Wang \cite{SW.2020} proved that the solution of \eqref{eq:1.5} does not blow up in finite time for a large relaxation enhancing flow ${\bf u}$, where the range of $\alpha$ needs to be restricted. Niu-Shi-Wang \cite{NSW.2025} proved that the blow-up of the solution of \eqref{eq:1.5} can be suppressed by a large shear flow ${\bf u}$ for $3/2<\alpha<2$. The suppression of the blow-up phenomenon through fluid motion for other related models could be found in \cite{CW.2024,DT.2024,FSW.2022,He.2018,He.2023,LXX.2025,SW.2023,ZZZ.2021}.

\vskip .05in

However, in most of the previous work, the $x$-variable is required to be in a periodic domain and this requirement ensures that the corresponding frequencies of $\Delta_x$ are discrete. Thus for the non-zero modes, there is a spectrum gap which helps us to capture the dissipation enhancement; and for the zero mode, there is no enhanced dissipation effect. If this requirement is removed and $x\in \mathbb{R}$ is considered, the corresponding continuous spectra make it hard to separate zero and non-zero modes as before. New methods and also new ideas are needed to capture the enhanced dissipation effect in this new setting. Recently, Coti Zelati-Gallay \cite{CZG.2023} considered the heat equation with a shear flow and $(x,y)\in\mathbb{R}\times [0,L]$ and proved that there is an enhanced dissipation effect and Taylor dispersion for high and low frequency parts respectively. For the whole space case, Deng-Shi-Wang \cite{DSW.2023} considered the classical Keller-Segel model with a Couette flow in $\mathbb{R}^3$ and proved that the blow-up in 3-D Keller-Segel model can be suppressed by a sufficiently large Couette flow; meanwhile, Arbon-Bedrossian \cite{AB.2024} and Li-Liu-Zhao \cite{LLZ.2025} studied the stability of 2-D Navier-Stokes equation near a Couette flow in $\mathbb{R}^2$ and gave the stability threshold.

\vskip .05in

In this paper, we consider the generalized Keller-Segel model in the background of a Couette flow with spatial variables $(x,y,z)\in\mathbb{R}^3$. We obtain the global existence and decaying rates of the solution by the Green's function method to overcome the difficulty caused by $x\in\mathbb{R}$. It is interesting that we prove even in 3-D the blow-up of the solution for the generalized Keller-Segel model could be suppressed by a Couette flow with a large enough amplitude in the whole space case. It reveals a totally different mechanism compared with the periodic case studied in \cite{BH.2017,He.2018,HR.2018,NSW.2025}. The suppression effect of a shear flow for the case $x\in\mathbb{R}$ is more like the one of a relaxation enhancing flow and ensures the global existence of solutions in any spatial dimension. This statement comes from a fact that our method used in this paper could be applied for any spatial dimension and thus similar results could be obtained. In fact, compare the Green's function with a Couette flow used in this paper and a classical fractional heat kernel and one could find that the differences lie only in $x$-variable and $y$-variable (and thus the construction and estimates of the Green's function with a Couette flow could be similar for any spatial dimension no less than 2). Thus, the method and estimates used here could also be applied for a general $\mathbb{R}^d (d\ge 2)$ case and one could prove a similar result that in the whole space case a sufficiently large Couette shear flow always leads to a strong enough dissipation enhancement which ensures the global existence of solutions.

\vskip .05in

The main result of this paper is as follows:

\begin{theorem}\label{the:1.1}
Let $\alpha\in(1, 2]$ and $p\in[2,\infty)$, for any non-negative initial function satisfying
\begin{equation}\label{2000}
n_{0}(x,y,z)\in W^{3, p} (\mathbb{R}^3)\cap L^{1}(\mathbb{R}^3),
\end{equation}
there exists a positive constant $A_{0}=A_0(\alpha, n_{0})$, such that for any $A\geq A_{0}$ the system \eqref{eq:1.1} has a unique classical solution satisfying
$$
n(x,y,z,t)\in C\big({{W^{3, p}}}(\mathbb{R}^3)\cap L^1(\mathbb{R}^3) ; \mathbb{R}^{+}\big)
$$
and
$$
\|D^{\vartheta}n(\cdot, \cdot, \cdot, t)\|_{L^p(\mathbb{R}^3)}\leq C(1+t)^{-\big(\frac{3}{\alpha}+1\big)\big(1-\frac{1}{p}\big)-\frac{|\vartheta|}{\alpha}},\ \  |\vartheta|\leq 3.
$$
\end{theorem}

\vskip .05in

The main tool used to prove Theorem \ref{the:1.1} is the Green's function of \eqref{eq:1.1} which is given by \eqref{eq:3.1}.
The construction of such a Green's function is highly nontrivial due to the variable coefficient $Ay$ in the linear equation in \eqref{eq:3.1}. In the previous study of Green's functions for time-evolution equations, the variable coefficients in linear level always cause great troubles and a generalized method is to approximate such Green's functions by the ones for linear equations with constant coefficients. For more details, one could refer to \cite{LWY.2009,NWY.2000} for approximated Green's functions of diffusion waves and \cite{Liu.1997,Liu.2009,Liu.2015,Yu.2010} for Green's functions of transversal fields in the study of shock profiles. However, it does not apply to the cases in which variable coefficients lead to new phenomena. The problem in this paper is just one of such cases, since one of the main ingredients of the dissipation enhancement is the variable coefficient $Ay$ (and it is also the Couette flow). For the very special case $\alpha=2$ (i.e. the classical Keller-Segel model with a Couette flow), there is an exact formula for the Green's function (one could refer to \cite{DSW.2023,MP.1977}); while for a generalized Keller-Segel model with a Couette flow and $0<\alpha<2$, there is no exact solution for the Green's function.

\vskip .05in

It seems difficult to construct the Green's function for the problem considered in this paper. However, we still believe that it is a natural choice and the reason is as follows. Except for some special cases, the global existence problems for most equations are independent of decaying rates of the solutions. In this paper, the suppression of blow-up is supposed to come from the enhanced dissipation and this expection shows a tight connection between those two aspects. Therefore a very important step in this paper is to establish the dissipation enhancement by obtaining the decaying rates of the solutions and Green's function should be one of the most efficient tools for whole space problem and to describe algebraic decaying structures.

\vskip .05in

The Green's function method for dissipation enhancement problems is at its very beginning. Recently, in \cite{DSW.2023} the pointwise structure of the solution for the classical Keller-Segel model with a Couette flow was obtained and thus both the enhanced dissipation and the wave structure induced by the classical heat kernel and a Couette flow are clear. It also reveals the main difficulty to attain dissipation enhancement and the key estimate required to overcome such a trouble in the technical level. It should be an estimate of the corresponding Green's function in a suitable norm containing no singularity at $t=0$ and such a key estimate is always the $L^1$ estimate. In this paper, it is fortunate that although there is no exact solution for the Green's function $\mathbb{G}$ defined by \eqref{eq:3.1}, there is an exact formula (one could refer to Lemma \ref{lem:3.1}) for its Fourier transform $\widehat{\mathbb{G}}$. The $L^p (2\leq p\leq\infty)$ estimates of the Green's function could be easily obtained from the $L^1$ and $L^2$ estimates of $\widehat{\mathbb{G}}$ and Young's inequality; however, it is nontrivial to get the $L^1$ estimate of the Green's function from its Fourier transform. Liu-Wang \cite{LW.1998} introduced an approach for pointwise estimates of functions from their Fourier transform based on micro-local analysis and pointwise structure can definitely yield the $L^1$ estimate. However, it does not work for the problem here since this approach is designed for low-frequency parts and could not reveal the difference between $t$ and $1+t$. Thus a direct use of this approach will lead to an $L^1$ estimate still containing singularities at $t=0$ if $\alpha$ in \eqref{eq:1.1} is not an even positive integer.

\vskip .05in

Roughly speaking, the symbol $\left|\vec\Xi\right|^\alpha$ (where one denotes the spectral variable $\Xi=(\xi, \eta, \zeta)\in\mathbb{R}^3$ as the dual variable of the space variable $(x, y, z)\in\mathbb{R}^3$ and $\alpha>0$) is not analytic unless $\alpha$ is an even positive integer, and the only single point is the original point (belonging to the low frequency region). However, an expected singularity at $t=0$ (similar to that for heat kernel) arises from the high frequency parts. Therefore, for the classical fractional heat kernel, the elimination of such singularities in $L^1$ norm could be fulfilled by combining Liu-Wang's approach and the low frequency-high frequency decomposition. In this paper, a fractional heat kernel together with the advection induced by a Couette flow is considered and information from different frequency regions could be mixed up. It causes new troubles compared with the classical fractional heat kernel. To overcome this trouble, we introduce a new decomposition where both spatial regions and frequency regions are involved and more details could be found in Lemma \ref{lem:3.6}.

\vskip .05in

The rest of this paper is arranged as follows. In Section \ref{sec.2}, we prepare analytic tools which will be used in the estimates of the Green's function and the proof of nonlinear stability. In Section \ref{sec.3}, the estimates of Green's function are obtained and the main trouble but also really essential part lies in its $L^1$ estimate. In Section \ref{sec.4}, we obtain a regularity criterion of the solution first; later based on the regularity criterion and estimates of the Green's function gained in Section \ref{sec.3}, we prove Theorem \ref{the:1.1} and establish the global well-posedness of 3-D generalized Keller-Segel equation in the background of a large Couette flow.

\section{Preliminaries }\label{sec.2}
Denote the spectral variable $\Xi=(\xi, \eta, \zeta)\in\mathbb{R}^3$ as the dual variable of the space variable $(x, y, z)\in\mathbb{R}^3$. In this section, we prepare some technical lemmas which will be used later. The first one is about the Young's inequality of the kernel function.

\begin{lemma}\label{lem:2.1}{(\cite{CDS.1993, SEM.1993})}
Let $z, z'\in\mathbb{R}^d$, if the kernel function $K(z, z')\in\mathbb{R}^{d}\times\mathbb{R}^d$ is a measurable function satisfying
$$
\|K(\cdot, z')\|_{L^q}\leq A,\ \|K(z, \cdot)\|_{L^q}\leq B.
$$
We denote the integral operator
$$
T f(z)=\int_{\mathbb{R}^d}K(z, z')f(z')\, dz',
$$
where $f(z)\in L^p(\mathbb{R}^d)$, then we have
$$
\|Tf\|_{L^r}\leq C\|f\|_{L^p}.
$$
Here $C=\max\{A, B\}$ and $1+\frac{1}{r}=\frac{1}{p}+\frac{1}{q},\ 1\leq p, q, r\leq\infty$. Furthermore, we have
$$
\|Tf\|_{L^r}\leq A^{\frac{q}{r}}B^{q-\frac{q}{p}}\|f\|_{L^p}.
$$
\end{lemma}

The following lemma is about the quasi-differential operators.

\begin{lemma}\label{lem:2.3}{{(\cite{GO.2014})}}
Assume $1<p<\infty,\ s>0$, then we have
$$
\|\Lambda^s(fg)(\cdot)-f(\cdot)\Lambda^s g(\cdot)\|_{L^p}\leq C(\|\Lambda f(\cdot)\|_{L^{p_{1}}}\|\Lambda^{s-1}g(\cdot)\|_{L^{p_{2}}}+\|\Lambda^s f(\cdot)\|_{L^{q_{1}}}\|g(\cdot)\|_{L^{q_{2}}}),
$$
and
$$
\|\Lambda^{s}(fg)(\cdot)\|_{L^p}\leq C(\|\Lambda^s f(\cdot)\|_{L^{p_{1}}}\|g(\cdot)\|_{L^{p_{2}}}+\|f(\cdot)\|_{L^{q_{1}}}\|\Lambda^{s}g(\cdot)\|_{L^{q_{2}}}),
$$
where
$$
\frac{1}{p}=\frac{1}{p_{1}}+\frac{1}{p_{2}}=\frac{1}{q_{1}}+\frac{1}{q_{2}}.
$$
\end{lemma}

Next two technical lemmas play a key role in the estimates of the Green's function in frequency domain.

\begin{lemma}\label{lem:3.2}
(1) For any $\widetilde{\alpha}\geq 0$, there exists $C_{\widetilde{\alpha}}>0$ such that
$$
\int_{0}^{t}|\eta+As\xi|^{\widetilde{\alpha}}\, ds\geq C_{\widetilde{\alpha}}\big(|\eta|^{\widetilde{\alpha}}+(At)^{\widetilde{\alpha}}|\xi|^{\widetilde{\alpha}}\big)t.
$$
(2) For any $\beta\in (-1,\ 0)$, there exists $C_{\beta}>0$ such that
\begin{equation}\label{eq:1003}
\int_{0}^{t}|\eta+As\xi|^{\beta}\, ds\leq C_{\beta}\big(|\eta|+(At)|\xi|\big)^{\beta}t.
\end{equation}
\end{lemma}

\begin{proof}
The proof of (1)  can be found in \cite[Lemma 3.1]{MX.2009} and we omit the details. Here, we only give the proof of (2). Setting $s=t\tau, \widetilde{\xi}=tA\xi$, for any $\beta\in(-1, 0)$, \eqref{eq:1003} is equivalent to the following inequality
$$
I\triangleq\int_{0}^{1}|\eta+\widetilde{\xi}\tau|^{\beta}\, d\tau\leq C_{\beta}\big(|\eta|+|\widetilde{\xi}|\big)^{\beta}.
$$
Assume $|\widetilde{\xi}|>0$ (if $|\widetilde{\xi}|=0$ the above inequality holds obviously) and then we obtain the estimate of $I$ in two cases: If $|\eta|\leq |\widetilde{\xi}|$, it is easy to get
$$
\begin{aligned}
I=&|\widetilde{\xi}|^\beta\int_{0}^{1}\Big|\frac{\eta}{\widetilde{\xi}}+\tau\Big|^{\beta}\, d\tau\leq |\widetilde{\xi}|^\beta\int_{0}^{1}\Big|\tau-\frac{|\eta|}{|\widetilde{\xi}|}\Big|^{\beta}\, d\tau\\
= &|\widetilde{\xi}|^\beta\int_{|\eta|/|\widetilde{\xi}|}^{1}\Big(\tau-\frac{|\eta|}{|\widetilde{\xi}|}\Big)^{\beta}\, d\tau+|\widetilde{\xi}|^\beta\int_{0}^{|\eta|/|\widetilde{\xi}|}\Big(\frac{|\eta|}{|\widetilde{\xi}|}-\tau\Big)^{\beta}\, d\tau\\
\leq &C|\widetilde{\xi}|^\beta\frac{1}{\beta+1}\max_{0\leq \sigma\leq 1}\big((1-\sigma)^{\beta+1}+\sigma^{\beta+1}\big)
\leq C_{\beta}\big(|\eta|+|\widetilde{\xi}|\big)^{\beta}.
\end{aligned}
$$
If $|\eta|> |\widetilde{\xi}|$, it holds that
$$
\begin{aligned}
I=&|\eta|^\beta\int_{0}^{1}\Big|1+\frac{\widetilde{\xi}}{\eta}\tau\Big|^{\beta}\, d\tau\leq |\eta|^\beta\int_{0}^{1}\Big|1-\tau\frac{|\widetilde{\xi}|}{|\eta|}\Big|^{\beta}\, d\tau\\
\leq &|\eta|^\beta\int_{0}^{1}\big|1-\tau\big|^{\beta}\, d\tau
\leq C_{\beta}\big(|\eta|+|\widetilde{\xi}|\big)^{\beta}.
\end{aligned}
$$
Combining the above two cases, we finish the proof of \eqref{eq:1003}.
\end{proof}

\begin{lemma}\label{lem:3.3}
For $A, t>0$ and $\xi,\eta\in\mathbb{R}$, there exists a positive constant $C$ such that
$$
\xi^2+(1+At)^{-2}\eta^2\le C\left(\xi^2+(\eta+At\xi)^2\right),
$$
and
$$
 \xi^2+(\eta+At\xi)^2\le C\left((1+At)^2\xi^2+\eta^2\right).
 $$
\end{lemma}

\begin{proof}
The second inequality in this lemma could be obtained directly by using Cauchy's inequality:
$$
\xi^2+(\eta+At\xi)^2\le (1+2A^2t^2)\xi^2+2\eta^2,
$$
which shows that the second inequality holds.

For the first inequality, we prove it by dividing the regions. If $At|\xi|\le |\eta|/2$, it holds that
$$
\left|\eta+At\xi\right|\ge |\eta|-At|\xi|\ge |\eta|/2
$$
and thus
$$
\xi^2+(\eta+At\xi)^2\ge \xi^2+\eta^2/4.
$$
If $At|\xi|>|\eta|/2$, one has that
$$
\xi^2+(\eta+At\xi)^2\ge \xi^2/2+\xi^2/2\ge \xi^2/2+\eta^2/(8A^2t^2).
$$
Above all, we finish the proof.
\end{proof}

The final lemma of this section is about the very important $L^1$ estimate and also shows a connection between a function and its Fourier transform in some suitable Sobolev spaces:

\begin{lemma}\label{lem:3.7n}
If $\widehat{f}(\xi, \eta, \zeta)\in H^2(\mathbb{R}^3)$, one has that
\begin{equation}\label{3002}
\|f(\cdot, \cdot, \cdot)\|_{L^1}\leq C\|\widehat{f}(\cdot, \cdot, \cdot)\|_{H^2}.
\end{equation}
\end{lemma}

\begin{proof}
One could introduce
$$
F(x, y, z, t)=\big(1+(x^2+y^2+z^2)\big)f(x, y, z, t),
$$
which together with Parseval formula yields that
\begin{equation*}
\begin{split}
\|f(\cdot, \cdot, \cdot, t)\|_{L^1}\leq&\left(\int_{\mathbb{R}^3}\Big(\frac{1}{1+(x^2+y^2+z^2)}\Big)^2\, dx\, dy\, dz\right)^{1/2}\left(\int_{\mathbb{R}^3} F^2\, dx\, dy\, dz\right)^{1/2}\\
\leq& C\left(\int_{\mathbb{R}^3}\big|\widehat{F}\big|^2\, d\xi\, d\eta\, d\zeta\right)^{1/2}=C\left(\int_{\mathbb{R}^3}\Big(\big(1-\Delta\big)\widehat{f}\Big)^2\, d\xi\, d\eta\, d\zeta\right)^{1/2}\\
\leq & C\|\widehat{f}(\cdot, \cdot, \cdot, t)\|_{H^2}
\end{split}
\end{equation*}
and thus obtains \eqref{3002}.
\end{proof}

\section{The estimates of Green's function}\label{sec.3}

In this section, we construct the Green's function for \eqref{eq:1.1} and obtain its $L^p (1\le p\le \infty)$ estimates. In this paper, the Green's function and its estimates serve as a powerful tool for $L^p\ (2\leq p<\infty)$ estimate of the solution which also guarantees its global existence.

\vskip .05in

The linearized equation of \eqref{eq:1.1} is a fractional heat equation with a Couette flow and it is a linearized equation with variable coefficients. The Green's function $\mathbb{G}(x,y,z,t;x',y',z')$ of \eqref{eq:1.1} is defined as follows:

\begin{equation}\label{eq:3.1}
\begin{cases}
\partial_{t}\mathbb{G}+Ay\partial_{x}\mathbb{G}+(-\Delta)^{\alpha/2}\mathbb{G}=0,\\
\mathbb{G}(x-x', y, z-z', 0; y')=\delta(x-x', y-y', z-z').
\end{cases}
\end{equation}
Here, $(x, y, z)\in\mathbb{R}^3,\ \delta(x, y, z)$ is a 3-dimensional Dirac-delta function. Apply the Fourier transform to \eqref{eq:3.1} and get
\begin{equation}\label{eq:3.2}
\begin{cases}
\partial_{t}\widehat{\mathbb{G}}-A\xi\partial_{\eta}\widehat{\mathbb{G}}+(\xi^2+\eta^2+\zeta^2)^{\alpha/2}\widehat{\mathbb{G}}=0,\\
\widehat{\mathbb{G}}(\xi, \eta, \zeta, 0; x', y', z')=\exp(-ix'\xi-iy'\eta-iz'\zeta).
\end{cases}
\end{equation}

It is always difficult to solve \eqref{eq:3.1} directly while it could be far more easier to obtain the representation of the Green's function in frequency domain:

\begin{lemma}\label{lem:3.1}
The solution of \eqref{eq:3.2} is
\begin{equation}\label{eq:1002}
\begin{split}
\widehat{\mathbb{G}}(\xi, \eta, \zeta, t; x', y', z')
=&\exp\Big(-ix'\xi-iy'(\eta+At\xi )-iz'\zeta\Big)\\
&\cdot \exp\left(-\int_{0}^{t}\big(\xi^2+(\eta+As\xi )^2+\zeta^2\big)^{\alpha/2}\, ds\right).
\end{split}
\end{equation}
\end{lemma}

\begin{proof}
To verify the solution representation \eqref{eq:1002}, one needs to show that \eqref{eq:1002} satisfies the equation i.e. $\eqref{eq:3.2}_{1}$ and also the initial condition i.e. $\eqref{eq:3.2}_{2}$. When $t=0$, it is obvious that
$$
\widehat{\mathbb{G}}(\xi, \eta, \zeta, 0; x', y', z')=\exp(-ix'\xi-iy'\eta-iz'\zeta),
$$
and thus \eqref{eq:1002} meets the initial condition. After straightforward computations, one also has that
$$
\partial_{t}\widehat{\mathbb{G}}(\xi, \eta, \zeta, t; x', y', z')
=\widehat{\mathbb{G}}(\xi, \eta, \zeta, t; x', y', z')\Big(-iA\xi y'-(\xi^2+(\eta+At \xi )^2+\zeta^2)^{\alpha/2}\Big)
$$
and
$$
\begin{aligned}
A\xi\partial_{\eta}\widehat{\mathbb{G}}(\xi, \eta, \zeta, t; x', y', z')
=&\widehat{\mathbb{G}}(\xi, \eta, \zeta, t; x', y', z')\\
&\cdot\Big(-iA\xi y'+(\xi^2+\eta^2+\zeta^2)^{\alpha/2}-(\xi^2+(\eta+At \xi )^2+\zeta^2)^{\alpha/2}\Big),
\end{aligned}
$$
which justify that $\widehat{\mathbb{G}}$ also satisfies the equation $\eqref{eq:3.2}_{1}$.
\end{proof}

\subsection{$L^p\ (2\leq p\leq\infty)$ estimates of the Green's function}

The solution representation \eqref{eq:1002} of the Fourier transform $\widehat{\mathbb{G}}$ gives the $L^p\ (2\leq p\leq\infty)$ estimate of the Green's function very directly and the key estimates of this part are $L^1$ and $L^2$ estimates of $\hat{\mathbb{G}}$:

\begin{lemma}\label{lem:3.4}
For $k=k_{1}+k_{2}+k_{3}$, one has that
\begin{equation}\label{eq:3.3}
\left\||\xi|^{k_{1}}|\eta|^{k_{2}}|\zeta|^{k_{3}}\widehat{\mathbb{G}}(\cdot, \cdot, \cdot, t; x', y', z')\right\|_{L^1}\leq Ct^{-\frac{3+k}{\alpha}}(1+At)^{-k_{1}-1}
\end{equation}
and
\begin{equation}\label{eq:3.4}
\left\||\xi|^{k_{1}}|\eta|^{k_{2}}|\zeta|^{k_{3}}\widehat{\mathbb{G}}(\cdot, \cdot, \cdot, t; x', y', z')\right\|_{L^2}\leq Ct^{-\frac{3+2k}{2\alpha}}(1+At)^{-k_{1}-\frac{1}{2}}.
\end{equation}
\end{lemma}

 \begin{proof}
 Lemma \ref{lem:3.2} together with Lemma \ref{lem:3.1} yields that
 \begin{equation}\label{eq:1004}
 \begin{split}
 \Big|\widehat{\mathbb{G}}(\xi, \eta, \zeta, t; x', y', z')\Big|\leq & \exp\left(-C\int_{0}^{t}\big(|\xi|^{\alpha}+|\eta+As \xi |^{\alpha}+|\zeta|^{\alpha}\big)\, ds\right)\\
 \leq&\exp\Big(-C\big(|\xi|^{\alpha}t(1+At)^{\alpha}+|\eta|^{\alpha}t+|\zeta|^{\alpha}t\big)\Big),
 \end{split}
 \end{equation}
 where we use the following inequality
 $$
\begin{aligned}
1/3\big(|\xi|^{\alpha}+|\eta+As \xi |^{\alpha}+|\zeta|^{\alpha}\big) \leq \big(\xi^2+(\eta+As \xi )^2+\zeta^2\big)^{\alpha/2}\leq \big(|\xi|^\alpha+|\eta+As \xi |^{\alpha}+|\zeta|^{\alpha}\big).
\end{aligned}
$$
Based on the inequality \eqref{eq:1004}, the $L^1$ and $L^2$ norm of $|\xi|^{k_{1}}|\eta|^{k_{2}}|\zeta|^{k_{3}}\widehat{\mathbb{G}}(x,y,z, t;x', y', z')$ can be estimated as follows:
 \begin{equation*}
 \begin{split}
& \left\||\xi|^{k_{1}}|\eta|^{k_{2}}|\zeta|^{k_{3}}\widehat{\mathbb{G}}(\cdot, \cdot, \cdot, t; x', y', z')\right\|_{L^1}\\
\leq & \int_{\mathbb{R}^3}|\xi|^{k_{1}}|\eta|^{k_{2}}|\zeta|^{k_{3}}\exp\Big(-C\big(|\xi|^{\alpha}t(1+At)^{\alpha}
 +|\eta|^{\alpha}t+|\zeta|^{\alpha}t\big)\Big)\, d\xi\, d\eta\, d\zeta\\
 \leq & Ct^{-\frac{3+k}{\alpha}}(1+At)^{-k_{1}-1}
 \end{split}
\end{equation*}
 and
\begin{equation*}
\begin{split}
&  \left\||\xi|^{k_{1}}|\eta|^{k_{2}}|\zeta|^{k_{3}}\widehat{\mathbb{G}}(\cdot, \cdot, \cdot, t; x', y', z')\right\|_{L^2}\\
\leq & \int_{\mathbb{R}^3}|\xi|^{2 k_{1}}|\eta|^{2 k_{2}}|\zeta|^{2 k_{3}}\exp\Big(-C\big(|\xi|^{\alpha}t(1+At)^{\alpha}+|\eta|^{\alpha}t+|\zeta|^{\alpha}t\big)\Big)\, d\xi\, d\eta\, d\zeta\\
\leq & Ct^{-\frac{3+2k}{2\alpha}}(1+At)^{-k_{1}-\frac{1}{2}},
\end{split}
\end{equation*}
which verify \eqref{eq:3.3} and \eqref{eq:3.4} respectively.
 \end{proof}

Now one could combine Parseval formula, Young's inequality and Lemma \ref{lem:3.4} to obtain the $L^p (2\leq p\leq\infty)$ estimates of $\mathbb{G}$:

\begin{lemma}\label{lem:3.5}
 For $p\in [2,\infty], k=k_{1}+k_{2}+k_{3}$, it holds that
 \begin{equation}\label{eq:3.5}
\left \|\partial_{x}^{k_{1}}\partial_{y}^{k_{2}}\partial_{z}^{k_{3}}\mathbb{G}(\cdot, \cdot, \cdot, t; x', y', z')\right\|_{L^p}\leq C t^{-\frac{3}{\alpha}\big(1-\frac{1}{p}\big)-\frac{k}{\alpha}}(1+At)^{-k_{1}-\big(1-\frac{1}{p}\big)}
 \end{equation}
 and
\begin{equation}\label{eq:3.6}
\left \|\partial_{x}^{k_{1}}\partial_{y}^{k_{2}}\partial_{z}^{k_{3}}\mathbb{G}(x, y, z, t; \cdot, \cdot, \cdot)\right\|_{L^p}\leq C t^{-\frac{3}{\alpha}\big(1-\frac{1}{p}\big)-\frac{k}{\alpha}}(1+At)^{-k_{1}-\big(1-\frac{1}{p}\big)}.
 \end{equation}
 \end{lemma}

\begin{proof}
Using the interpolation inequality, Parseval formula, Young's inequality and Lemma \ref{lem:3.4}, one has that
 \begin{equation*}
 \begin{split}
&\left\|\partial_{x}^{k_{1}}\partial_{y}^{k_{2}}\partial_{z}^{k_{3}}\mathbb{G}(\cdot, \cdot, \cdot, t; x', y', z')\right\|_{L^p}\\
 \leq&C\left\|\partial_{x}^{k_{1}}\partial_{y}^{k_{2}}\partial_{z}^{k_{3}}\mathbb{G}(\cdot, \cdot, \cdot, t; x', y', z')\right\|_{L^2}^{2/p}
 \cdot\left\|\partial_{x}^{k_{1}}\partial_{y}^{k_{2}}\partial_{z}^{k_{3}}\mathbb{G}(\cdot, \cdot, \cdot, t; x', y', z')\right\|_{L^{\infty}}^{1-2/p}\\
 \leq&C\left\||\xi|^{k_{1}}|\eta|^{k_{2}}|\zeta|^{k_{3}}\widehat{\mathbb{G}}(\cdot, \cdot, \cdot, t; x', y', z')\right\|_{L^2}^{2/p}\cdot\left\||\xi|^{k_{1}}|\eta|^{k_{2}}|\zeta|^{k_{3}}\widehat{\mathbb{G}}(\cdot, \cdot, \cdot, t; x', y', z')\right\|_{L^1}^{1-2/p}\\
 \leq&Ct^{-\frac{3}{\alpha}\big(1-\frac{1}{p}\big)-\frac{k}{\alpha}}(1+At)^{-k_{1}-\big(1-\frac{1}{p}\big)},
 \end{split}
 \end{equation*}
 which yields \eqref{eq:3.5}.

 The proof of \eqref{eq:3.6} is based on the relationship between the variables $(x, y, z)$ and $(x', y', z')$: Since one has that
 \begin{equation}\label{eq:3.9}
 \begin{split}
 &\mathbb{G}(x, y, z, t; x', y', z')\\
 =&\int_{\mathbb{R}^3}e^{i(x\xi+y\eta+z\zeta)}\widehat{\mathbb{G}}(\xi, \eta, \zeta, t; x', y', z')\, d\xi\, d\eta\ d\zeta\\
 =&\int_{\mathbb{R}^3}e^{i\big((x-x'-Aty')\xi+(y-y')\eta+(z-z')\zeta\big)}\exp\left(-\int_{0}^{t}\big(\xi^2+(\eta+As \xi )^2+\zeta^2\big)^{\alpha/2}\, ds\right)\, d\xi\, d\eta\, d\zeta\\
 =&F(x-x'-Aty', y-y', z-z', t),
 \end{split}
 \end{equation}
 the derivatives and also integrals about $(x', y', z')$ could be transferred to those about $(x, y, z)$ and thus one could obtain \eqref{eq:3.6}  from \eqref{eq:3.5}.
 \end{proof}

The estimates in Lemma \ref{lem:3.5} contain singularities at $t=0$ and this fact may cause troubles in the later closure of nonlinearity. One still looks for an estimate without such kind of singularities.

\subsection{$L^1$ estimate of the Green's function} If one would like to show that Lemma \ref{lem:3.5} is also correct for $p<2$, more details about the Fourier transform $\widehat{\mathbb{G}}$ would be required: Denote $\widehat{\mathbb{G}}=\widehat{\mathbb{G}}_{1}\widehat{\mathbb{G}}_{2}$ with
$$
\begin{aligned}
\widehat{\mathbb{G}}_{1}(\xi, \eta, \zeta, t; x', y', z')&=\exp\big(-ix'\xi-iy'(\eta+At \xi)-iz'\zeta\big),\\
\widehat{\mathbb{G}}_{2}(\xi, \eta, \zeta, t)&=\exp\left(-\int_{0}^{t}\big(\xi^2+(\eta+As \xi )^2+\zeta^2\big)^{\alpha/2}\, ds\right).
\end{aligned}
$$

\begin{lemma}\label{lem:3.6}
For $\alpha\in(1, 2]$ and $k=k_1+k_2+k_3=1$ or $2$, there exists a positive constant $C$ such that
$$
\left\|\partial_{x}^{k_{1}}\partial_{y}^{k_2}\partial_{z}^{k_{3}}\mathbb{G}_{2}(\cdot, \cdot, \cdot, t)\right\|_{L^1}\leq Ct^{-\frac{k}{\alpha}}(1+At)^{-k_{1}}.
$$
\end{lemma}

\begin{proof}
First we consider $k=1$. Denote
$$
\theta(\xi, \eta, \zeta, t)\equiv\xi^2+(\eta+At \xi)^2+\zeta^2
$$
and
$$
\mathbb{H}(\xi, \eta, \zeta, t)\equiv\int^t_0\big(\theta(\xi, \eta, \zeta, s)\big)^{\alpha/2}\ ds.
$$
Since $k=k_{1}+k_{2}+k_{3}=1$ and $k_{i}\, (i=1, 2, 3)$ are integers, there are three possible combinations of $k_{i}\, (i=1, 2, 3)$ and one considers each respectively.\\

\textbf{Case 1. The estimate of $\left\|\partial_{x}\mathbb{G}_{2}(\cdot, \cdot, \cdot, t)\right\|_{L^1}$}.\\

From the definition of $\theta(\xi, \eta, \zeta, t)$ and $\mathbb{H}(\xi, \eta, \zeta, t)$, one has that
\begin{equation}\label{1005}
\Big|\partial^2_{\xi}\big(\xi\widehat{\mathbb{G}}_{2}(\xi, \eta, \zeta, t)\big)\Big|
\leq C\Big(|\xi|(\mathbb{H}^\prime_\xi)^2+|\xi||\mathbb{H}^{\prime\prime}_\xi|+|\mathbb{H}^\prime_\xi|\Big)\widehat{\mathbb{G}}_{2}(\xi, \eta, \zeta, t),
\end{equation}
and then one could combine Lemma \ref{lem:3.3} to obtain that
\begin{equation}\label{4001}
\begin{split}
\mathbb{H}^\prime_\xi(\xi, \eta, \zeta, t)=&\alpha\int^t_0\big(\xi^2+(\eta+As \xi )^2+\zeta^2\big)^{\alpha/2-1}\big(\xi+(\eta+As\xi)As\big)ds\\
\leq & C\big(\xi^2(1+At)^2t^{2/\alpha}+\eta^2t^{2/\alpha}
+\zeta^2t^{2/\alpha}\big)^{(\alpha-1)/2}t^{1/\alpha}(1+At),
\end{split}
\end{equation}

\begin{equation}\label{4002}
\begin{aligned}
|\xi|\big(\mathbb{H}^\prime_\xi(\xi, \eta, \zeta, t)\big)^2\leq & C|\xi|\big(\xi^2(1+At)^2t^{2/\alpha}+\eta^2t^{2/\alpha}+\zeta^2t^{2/\alpha}\big)^{\alpha-1}t^{2/\alpha}(1+At)^2\\
\leq &C\big(\xi^2(1+At)^2t^{2/\alpha}+\eta^2t^{2/\alpha}
+\zeta^2t^{2/\alpha}\big)^{\alpha-1/2}t^{1/\alpha}(1+At).
\end{aligned}
\end{equation}
The second derivative with respect to $\xi$ of $\mathbb{H}(\xi, \eta, \zeta, t)$ could be computed similarly:
\begin{equation}\label{3000}
\begin{split}
|\xi|\mathbb{H}^{\prime\prime}_\xi(\xi, \eta, \zeta, t)=&\alpha(\alpha-2)\int^t_0|\xi|\big(\xi^2+(\eta+As \xi )^2+\zeta^2\big)^{\alpha/2-2}\big(\xi+(\eta+As\xi)As\big)^2\, ds\\
&+\alpha\int_{0}^{t}|\xi|\big(\xi^2+(\eta+As\xi)^2+\zeta^2\big)^{\alpha/2-1}\big(1+(As)^2\big)\, ds\\
\leq & C\int^t_0|\xi|\big(\xi^2+(\eta+As \xi )^2+\zeta^2\big)^{\alpha/2-1}\, ds(1+At)^2.
\end{split}
\end{equation}
In the inequality \eqref{3000}, if $At\leq 1$, one has that
$$
\begin{aligned}
|\xi|\mathbb{H}^{\prime\prime}_\xi(\xi, \eta, \zeta, t)
\leq& C\int^t_0|\xi|^{\alpha-1}\, ds(1+At)^2\leq C\big(|\xi|(1+At)t^{1/\alpha}\big)^{\alpha-1}t^{1/\alpha}(1+At).
\end{aligned}
$$
If $At> 1$,  since $\alpha/2-1\leq 0$, one could apply Lemma \ref{lem:3.2} (2) to obtain that
$$
\begin{aligned}
|\xi|\mathbb{H}^{\prime\prime}_\xi(\xi, \eta, \zeta, t)
\leq &|\xi|\int^t_0(\eta+As \xi )^{\alpha-2}\, ds(1+At)^2\leq Ct|\xi|\big(|\eta|+(At)|\xi|\big)^{\alpha-2}(1+At)^2\\
\leq &Ct|\xi|\big((1+At)|\xi|\big)^{\alpha-2}(1+At)^2\leq C\big((1+At)t^{1/\alpha}|\xi|\big)^{\alpha-1}t^{1/\alpha}(1+At).
\end{aligned}
$$
Combining the above estimates and \eqref{1005}, it holds that
\begin{equation*}
\begin{split}
\Big|\partial^2_{\xi}\big(\xi\widehat{\mathbb{G}}_{2}(\xi, \eta, \zeta, t)\big)\Big|\leq & C\Big(\big(\xi^2(1+At)^2t^{2/\alpha}+\eta^2t^{2/\alpha}
+\zeta^2t^{2/\alpha}\big)^{(\alpha-1)/2}\\
& +\big(\xi^2(1+At)^2t^{2/\alpha}+\eta^2t^{2/\alpha}
+\zeta^2t^{2/\alpha}\big)^{\alpha-1/2}\Big)t^{1/\alpha}(1+At)\widehat{\mathbb{G}}_{2}.
\end{split}
\end{equation*}
Thus one denotes
$$
\mathbb{Q}(\xi, \eta, \zeta, t)\equiv\big(\xi^2(1+At)^2t^{2/\alpha}+\eta^2t^{2/\alpha}+\zeta^{2}t^{2/\alpha}\big)
$$
and
$$
 \mathbb{E}(\xi, \eta, \zeta, t)\equiv\exp\Big(-C_{0}\big(|\xi|^{\alpha}t(1+(At)^{\alpha})+|\eta|^{\alpha}t+|\zeta|^{\alpha}t\big)\Big),
$$
where $C_{0}$ is chosen satisfying $\widehat{\mathbb{G}}_{2}\leq\mathbb{E}^2$. For any $\theta\in[0, 2]$, it is obvious that $\mathbb{Q}^{\theta}\mathbb{E}\leq C$. Thus, one has that
\begin{equation}\label{eq:3.10}
\Big|\partial^2_{\xi}\big(\xi\widehat{\mathbb{G}}_{2}(\xi, \eta, \zeta, t)\big)\Big|\leq C t^{1/\alpha}(1+At)\mathbb{E}(\xi, \eta, \zeta, t).
\end{equation}

Then consider $\Big|\partial^2_{\eta}\big(\xi\widehat{\mathbb{G}}_{2}(\xi, \eta, \zeta, t)\big)\Big|$:
\begin{equation}\label{1006}
\Big|\partial_{\eta}^2\big(\xi\widehat{\mathbb{G}}_{2}(\xi, \eta, \zeta, t)\big)\Big|\leq C \Big(|\xi|(\mathbb{H}_{\eta}^{\prime})^2+|\xi||\mathbb{H}_{\eta}^{\prime\prime}|\Big)\widehat{\mathbb{G}}_{2}(\xi, \eta, \zeta, t),
\end{equation}
and Lemma \ref{lem:3.3} results in that
\begin{equation*}
\begin{split}
\mathbb{H}_{\eta}^{\prime}(\xi,\eta,\zeta, t)=&\alpha\int_{0}^{t}\big(\xi^2+(\eta+As \xi )^2+\zeta^2\big)^{\alpha/2-1}(\eta+As \xi )\, ds\\
\leq&C\big(\xi^2(1+At)^2t^{2/\alpha}+\eta^2 t^{2/\alpha}+\zeta^2t^{2/\alpha}\big)^{(\alpha-1)/2}t^{1/\alpha}.
\end{split}
\end{equation*}
\begin{equation}\label{4010}
\begin{split}
|\xi|\big(\mathbb{H}_{\eta}^{\prime}(\xi, \eta, \zeta, t)\big)^2\leq&C|\xi|\big(\xi^2(1+At)^2t^{2/\alpha}+\eta^2t^{2/\alpha}+\zeta^2t^{2/\alpha}\big)^{\alpha-1}t^{2/\alpha}\\
\leq&C\big(\xi^2(1+At)^2t^{2/\alpha}+\eta^2t^{2/\alpha}+\zeta^2t^{2/\alpha}\big)^{\alpha-1/2}t^{1/\alpha}(1+At)^{-1},
\end{split}
\end{equation}
\begin{equation}\label{3001}
\begin{split}
|\xi|\mathbb{H}_{\eta}^{\prime\prime}(\xi, \eta,\zeta, t)= & \alpha(\alpha-2)|\xi|\int_{0}^{t}\big(\xi^2+(\eta+As \xi )^2+\zeta^2\big)^{\alpha/2-2}(\eta+As \xi )^2\, ds\\
&+\alpha|\xi|\int_{0}^{t}\big(\xi^2+(\eta+As \xi )^2+\zeta^2\big)^{\alpha/2-1}\, ds\\
\leq&C\int_{0}^{t}|\xi|\big(\xi^2+(\eta+As \xi )^2+\zeta^2\big)^{\alpha/2-1}\, ds.
\end{split}
\end{equation}
In the above inequality \eqref{3001}, similar to the estimates for \eqref{3000}, for both cases $At\leq 1$ and $At>1$, it holds that
\begin{equation*}
|\xi|\mathbb{H}_{\eta}^{\prime\prime}(\xi, \eta, \zeta, t)\leq C\big(|\xi|(1+At)t^{1/\alpha}\big)^{\alpha-1}t^{1/\alpha}(1+At)^{-1}.
\end{equation*}
Substituting all the above estimates of $|\xi|(\mathbb{H}_{\eta}^{\prime})^2$ and $|\xi||\mathbb{H}_{\eta}^{\prime\prime}|$ into \eqref{1006}, one gets that
\begin{equation}\label{1000}
\Big|\partial_{\eta}^2\big(\xi\widehat{\mathbb{G}}_{2}(\xi, \eta, \zeta, t)\big)\Big|\leq Ct^{1/\alpha}(1+At)^{-1}\mathbb{E}(\xi, \eta, \zeta, t),
\end{equation}
and by the same way, one could obtain that
\begin{equation}\label{1001}
\Big|\partial^2_{\zeta}\big(\xi\widehat{\mathbb{G}}_{2}(\xi, \eta, \zeta, t)\big)\Big|\leq C t^{1/\alpha}(1+At)^{-1}\mathbb{E}(\xi, \eta, \zeta, t).
\end{equation}

The next key step is to show the relationship between $\mathbb{G}_{2}(x, y, z, t)$ and the Fourier transformation $\widehat{\mathbb{G}}_{2}(\xi, \eta, \zeta, t)$: Denote
\begin{equation}\label{eq:3.12}
\big|\partial_{x}\mathbb{G}_{2}(x, y, z, t)\big|=C\Bigg|\int_{\mathbb{R}^3}e^{i(x\xi+y\eta+z\zeta)}\big(\xi\widehat{\mathbb{G}}_{2}(\xi, \eta, \zeta, t)\big)\, d\xi\, d\eta\, d\zeta\Bigg|\triangleq C\big|F_1(x, y, z, t)\big|,
\end{equation}

\begin{equation}\label{eq:3.13}
\big|x^{2}\partial_{x}\mathbb{G}_{2}(x, y, z, t)\big|=C\Bigg|\int_{\mathbb{R}^3}e^{i(x\xi+y\eta+z\zeta)}\partial_{\xi}^{2}\big(\xi\widehat{\mathbb{G}}_{2}(\xi, \eta, \zeta, t)\big)\, d\xi\, d\eta\, d\zeta\Bigg|\triangleq C\big|F_2(x, y, z, t)\big|
\end{equation}
and
\begin{equation}\label{eq:3.14}
\begin{split}
\big|(y^2+z^2)\big(\partial_x\mathbb{G}_{2}(x, y, z, t)\big)\big| & =C\Bigg|\int_{\mathbb{R}^3}e^{i(x\xi+y\eta+z\zeta)}\big(\partial_{\eta}^2+\partial_{\zeta}^2\big)\big(\xi\widehat{\mathbb{G}}_{2}(\xi, \eta, \zeta, t)\big)\, d\xi\, d\eta\, d\zeta\Bigg|\\
& \triangleq C\big|F_3(x, y, z, t)\big|.
\end{split}
\end{equation}
Divide the $\mathbb{R}^3$ into four regions: $\mathbb{R}^3=\Omega_{1}\cup \Omega_{2}\cup \Omega_{3}\cup \Omega_{4}$, with
$$
\begin{aligned}
\Omega_{1}=\Big\{x^2\leq t^{2/\alpha}(1+At)^2, y^2+z^2\leq t^{2/\alpha}\Big\},\\
\Omega_{2}=\Big\{x^2\leq t^{2/\alpha}(1+At)^2, y^2+z^2>t^{2/\alpha}\Big\},\\
\Omega_{3}=\Big\{x^2>t^{2/\alpha}(1+At)^2, y^2+z^2\leq t^{2/\alpha}\Big\},\\
\Omega_{4}=\Big\{x^2>t^{2/\alpha}(1+At)^2, y^2+z^2>t^{2/\alpha}\Big\},
\end{aligned}
$$
and obviously it holds that
\begin{equation}\label{eq:3.15}
1+\frac{x^2}{t^{2/\alpha}(1+At)^{2}}+\frac{y^2+z^2}{t^{2/\alpha}}\leq 3
\begin{cases}
1, & (x, y, z)\in \Omega_{1},\\
\frac{y^2+z^2}{t^{2/\alpha}}, &(x, y, z)\in \Omega_{2},\\
\frac{x^2}{t^{2/\alpha}(1+At)^2}, & (x, y, z)\in \Omega_{3},\\
\frac{x^2}{t^{2/\alpha}(1+At)^2}+\frac{y^2+z^2}{t^{2/\alpha}}, & (x, y, z)\in \Omega_{4}.
\end{cases}
\end{equation}
Setting $\mathbb{B}(x, y, z, t)\equiv\Big(1+\frac{x^2}{(1+At)^{2}t^{2/\alpha}}+\frac{(y^2+z^2)}{t^{2/\alpha}}\Big)^{-2}$, and combining Lemma \ref{lem:3.2}-\ref{lem:3.3}, one has the following estimates for $\mathbb{B}(x, y, z, t)$ and $\mathbb{E}(\xi, \eta, \zeta, t)$:
\begin{equation}\label{eq:3.17}
\begin{aligned}
\|\mathbb{B}(\cdot, \cdot, \cdot, t)\|^{1/2}_{L^1}\leq &Ct^{3/(2\alpha)}(1+At)^{1/2},\ \ \
\|\mathbb{E}(\cdot, \cdot, \cdot, t)\|_{L^2}
\leq Ct^{-3/(2\alpha)}(1+At)^{-1/2}.
\end{aligned}
\end{equation}
It combined with \eqref{eq:3.12}-\eqref{eq:3.14}, the inequalities of \eqref{eq:3.10}, \eqref{1000}, \eqref{1001} and \eqref{eq:3.17}, Lemma \ref{lem:3.5} and Young's inequality, yields that
\begin{equation}\label{eq:3.16}
\begin{aligned}
& \|\partial_x \mathbb{G}_{2}(\cdot, \cdot, \cdot, t)\|_{L^1}\\
\leq & C\big(\|\mathbb{B}(\cdot, \cdot, \cdot, t)\|^{1/2}_{L^1}\|F_1(\cdot, \cdot, \cdot, t)\|_{L^2}+\|\mathbb{B}(\cdot, \cdot, \cdot, t)\|^{1/2}_{L^1}\|\big((1+At)^{-2}t^{-2/\alpha}F_2+t^{-2/\alpha}F_3\big)(\cdot, \cdot, \cdot, t)\|_{L^2}\big)\\
\leq & C(1+At)^{1/2}t^{3/(2\alpha)}\|\xi\widehat{\mathbb{G}}_{2}(\cdot, \cdot, \cdot, t)\|_{L^2}\\
&+C t^{-2/\alpha}\big((1+At)^{-2}\|\partial_{\xi}^2(\xi\widehat{\mathbb{G}}_{2})(\cdot, \cdot, \cdot, t)\|_{L^2}+
\|(\partial_{\eta}^2+\partial_{\zeta}^2)(\xi \widehat{\mathbb{G}}_{2})(\cdot, \cdot, \cdot, t)\|_{L^2}\big)\|\mathbb{B}(\cdot, \cdot, \cdot, t)\|_{L^1}^{1/2}\\
\leq & C t^{-1/\alpha}(1+At)^{-1}\\
&+C \big(t^{-1/\alpha}(1+At)^{-1}\|\mathbb{E}(\cdot, \cdot, \cdot, t)\|_{L^2}+t^{-1/\alpha}(1+At)^{-1}\|\mathbb{E}(\cdot, \cdot, \cdot, t)\|_{L^2}\big)\|\mathbb{B}(\cdot, \cdot, \cdot, t)\|_{L^1}^{1/2}\\
\leq & Ct^{-1/\alpha}(1+At)^{-1}.
\end{aligned}
\end{equation}\\

\textbf{Case 2. The estimates of $\|\partial_y \mathbb{G}_{2}(\cdot, \cdot, \cdot, t)\|_{L^1}$ and $\|\partial_z \mathbb{G}_{2}(\cdot, \cdot, \cdot, t)\|_{L^1}$}.\\

Next, we consider $\partial_y\mathbb{G}_{2}(x, y, z, t)$: One could combine
$$
\Big|\partial^2_{\xi}\big(\eta\widehat{\mathbb{G}}_{2}(\xi, \eta, \zeta, t)\big)\Big|
\leq C\Big(|\eta|(\mathbb{H}^\prime_\xi)^2+|\eta||\mathbb{H}^{\prime\prime}_\xi|\Big)\widehat{\mathbb{G}}_{2}(\xi, \eta, \zeta, t),
$$
with
$$
\begin{aligned}
|\eta|\big(\mathbb{H}^\prime_\xi(\xi, \eta, \zeta, t)\big)^2=&\alpha^2|\eta|\Big(\int^t_0\big(\xi^2+(\eta+As \xi )^2+\zeta^2\big)^{\alpha/2-1}\big(\xi+(\eta+As\xi)As\big)ds\Big)^2\\
\leq &C\big(|\eta|t^{1/\alpha}\big)\big(\xi^2(1+At)^2t^{2/\alpha}+\eta^2t^{2/\alpha}+\zeta^2t^{2/\alpha}\big)^{\alpha-1}t^{1/\alpha}(1+At)^2,
\end{aligned}
$$
and
\begin{equation*}
\begin{split}
|\eta||\mathbb{H}^{\prime\prime}_\xi(\xi, \eta, \zeta, t)|=&\alpha(\alpha-2)\int^t_0|\eta|\big(\xi^2+(\eta+As \xi )^2+\zeta^2\big)^{\alpha/2-2}\big(|\xi|+|\eta+As\xi|As\big)^2ds\\
&+\alpha\int_{0}^{t}|\eta|\big(\xi^2+(\eta+As \xi )^2+\zeta^2\big)^{\alpha/2-1}(1+(As)^2)\, ds\\
\leq& C \int^t_0|\eta|\big(\xi^2+(\eta+As \xi )^2+\zeta^2\big)^{\alpha/2-1}\, ds(1+At)^2\\
\leq& C|\eta|\big(|\xi|(A t)+|\eta|\big)^{\alpha-2}t(1+At)^2\leq C(|\eta|t^{1/\alpha})^{\alpha-1}t^{1/\alpha}(1+At)^2,
\end{split}
\end{equation*}
to obtain that
\begin{equation}\label{eq:3.21}
\Big|\partial^2_{\xi}\big(\eta\widehat{\mathbb{G}}_{2}(\xi, \eta, \zeta, t)\big)\Big|\leq C t^{1/\alpha}(1+At)^{2}\mathbb{E}(\xi, \eta, \zeta, t).
\end{equation}
Similarly
$$
\begin{aligned}
\mathbb{H}_{\eta}^\prime(\xi, \eta,\zeta, t)= & \alpha\int_{0}^{t}\big(\xi^2+(\eta+As \xi )^2+\zeta^2\big)^{\alpha/2-1}(\eta+As\xi)\, ds\\
\leq & C\big(\xi^2(1+At)^2t^{2/\alpha}+\eta^2t^{2/\alpha}+\zeta^2t^{2/\alpha}\big)^{(\alpha-1)/2}t^{1/\alpha},
\end{aligned}
$$
$$
\begin{aligned}
|\eta|\big(\mathbb{H}^\prime_\eta(\xi, \eta, \zeta, t)\big)^2=& \alpha^2 |\eta|\Big(\int^t_0\big(\xi^2+(\eta+As \xi )^2+\zeta^2\big)^{\alpha/2-1}\big(\eta+As \xi \big)\, ds\Big)^2\\
\leq &C\big(|\eta|t^{1/\alpha}\big)\big(\xi^2(1+At)^2t^{2/\alpha}
+\eta^2t^{2/\alpha}+\zeta^2t^{2/\alpha}\big)^{\alpha-1}t^{1/\alpha}
\end{aligned}
$$
and
$$
\begin{aligned}
|\eta|\big|\mathbb{H}^{\prime\prime}_\eta(\xi, \eta, \zeta, t)\big|=&\alpha(\alpha-2)\int^t_0|\eta|\big(\xi^2+(\eta+As \xi )^2+\zeta^2\big)^{\alpha/2-2}|\eta+As \xi|^2\, ds\\
&+\alpha\int_{0}^{t} |\eta|\big(\xi^2+(\eta+As \xi )^2+\zeta^2\big)^{\alpha/2-1}\, ds\\
\leq& C\int_{0}^{t}|\eta|\big(\xi^2+(\eta+As \xi )^2+\zeta^2\big)^{\alpha/2-1}\, ds\leq C(|\eta|t^{1/\alpha})^{\alpha-1}t^{1/\alpha},
\end{aligned}
$$
results in that
\begin{equation}\label{eq:3.22}
\Big|\partial^2_{\eta}\big(\eta\widehat{\mathbb{G}}_{2}(\xi, \eta, \zeta, t)\big)\Big|\leq C\big(|\mathbb{H}_{\eta}^{\prime}|+|\eta||\mathbb{H}_{\eta}^{\prime\prime}|+|\eta|(\mathbb{H}_{\eta}^{\prime})^2\big)\leq Ct^{1/\alpha}\mathbb{E}(\xi, \eta, \zeta, t).
\end{equation}
The derivatives w.r.t. $\zeta$ could be obtained in a same way:
\begin{equation}\label{eq:3.23}
\Big|\partial^2_{\zeta}\big(\eta\widehat{\mathbb{G}}_{2}(\xi, \eta, \zeta, t)\big)\Big|\leq Ct^{1/\alpha}\mathbb{E}(\xi, \eta, \zeta, t).
\end{equation}

Denote
\begin{equation}\label{eq:3.24}
\begin{aligned}
\big|\partial_{y}\mathbb{G}_{2}(x, y, z, t)\big|=&C\Bigg|\int_{\mathbb{R}^3}e^{i(x\xi+y\eta+z\zeta)}\big(\eta\widehat{\mathbb{G}}_{2}(\xi, \eta, \zeta, t)\big)\, d\xi\, d\eta\, d\zeta\Bigg|\triangleq C\big|K_1(x, y, z, t)\big|,\\
\big|x^{2}\partial_{y}\mathbb{G}_{2}(x, y, z, t)\big|=&C\Bigg|\int_{\mathbb{R}^3}e^{i(x\xi+y\eta+z\zeta)}\partial_{\xi}^{2}\big(\eta\widehat{\mathbb{G}}_{2}(\xi, \eta, \zeta, t)\big)\, d\xi\, d\eta\, d\zeta\Bigg|\triangleq C\big|K_2(x, y, z, t)\big|,\\
\big|(y^2+z^2)\big(\partial_y\mathbb{G}_{2}(x, y, z, t)\big)\big|=&C\Bigg|\int_{\mathbb{R}^3}e^{i(x\xi+y\eta+z\zeta)}(\partial_{\eta}^2+\partial_{\zeta}^2)\big(\eta\widehat{\mathbb{G}}_{2}(\xi, \eta, \zeta, t)\big)\, d\xi\, d\eta\, d\zeta\Bigg|\\
\triangleq &C\big|K_3(x, y, z, t)\big|,
\end{aligned}
\end{equation}
and one could combine \eqref{eq:3.21}--\eqref{eq:3.24} to obtain that
\begin{equation}\label{eq:3.25}
\begin{aligned}
\|\partial_y \mathbb{G}_{2}(\cdot, \cdot, \cdot, t)\|_{L^1}\leq& C\|\mathbb{B}(\cdot, \cdot, \cdot, t)\|^{1/2}_{L^1}\|K_1(\cdot, \cdot, \cdot, t)\|_{L^2}\\
& + C\big(\|\mathbb{B}(\cdot, \cdot, \cdot, t)\|^{1/2}_{L^1}\|\big((1+At)^{-2}t^{-2/\alpha} K_2+t^{-2/\alpha} K_3\big)(\cdot, \cdot, \cdot, t)\|_{L^2}\big)\\
\triangleq& J_1+J_2.
\end{aligned}
\end{equation}
By Lemma \ref{lem:3.5} and \eqref{eq:3.21}-\eqref{eq:3.23}, one has that
\begin{equation}\label{eq:3.26}
J_1\leq C(1+At)^{1/2}t^{3/(2\alpha)}\|\eta\widehat{\mathbb{G}}_{2}(\cdot, \cdot, \cdot, t)\|_{L^2}\leq C t^{-1/\alpha},
\end{equation}
\begin{equation}\label{eq:3.27}
\begin{aligned}
J_2\leq &C\|\mathbb{B}(\cdot, \cdot, \cdot, t)\|^{1/2}_{L^1}t^{-2/\alpha}\big((1+At)^{-2}\|\partial_{\xi}^2(\eta\widehat{\mathbb{G}}_{2})(\cdot, \cdot, \cdot, t)\|_{L^2}+
\|(\partial_{\eta}^2+\partial_{\zeta}^2)(\eta\widehat{\mathbb{G}}_{2})(\cdot, \cdot, \cdot, t)\|_{L^2})\big)\\
\leq &C\|\mathbb{B}(\cdot, \cdot, \cdot, t)\|^{1/2}_{L^1}t^{-2/\alpha}\big(t^{1/\alpha}\|\mathbb{E}(\cdot, \cdot, \cdot, t)\|_{L^2}\big)\leq Ct^{-1/\alpha}.
\end{aligned}
\end{equation}
And thus \eqref{eq:3.25} together with \eqref{eq:3.26}-\eqref{eq:3.27} yields that
\begin{equation}\label{eq:3.28}
\|\partial_y \mathbb{G}_{2}(\cdot, \cdot, \cdot, t)\|_{L^1}\leq Ct^{-1/\alpha}.
\end{equation}

Similarly, one could prove that
\begin{equation}\label{eq:3.29}
\|\partial_z \mathbb{G}_{2}(\cdot, \cdot, \cdot, t)\|_{L^1}\leq Ct^{-1/\alpha}.
\end{equation}
Thus we finish the proof of this lemma for $k=1$ from \eqref{eq:3.16}, \eqref{eq:3.28} and \eqref{eq:3.29}.\\

Then consider $k=2$ which contains six cases:\\

\textbf{Case 1. The estimate of $\|\partial_{x}^2\mathbb{G}_{2}(\cdot, \cdot, \cdot, t)\|_{L^1}$}.\\

From the definition of $\theta(\xi, \eta, \zeta, t)$ and $\mathbb{H}(\xi, \eta, \zeta, t)$, one has that
\begin{equation}\label{4000}
\Big|\partial_{\xi}^2\big(\xi^2\widehat{\mathbb{G}}_{2}(\xi, \eta, \zeta, t)\big)\Big|\leq C\Big(1+|\xi||\mathbb{H}_{\xi}^\prime|+|\xi|^2|\mathbb{H}_{\xi}^{\prime\prime}|+|\xi|^2|\mathbb{H}_{\xi}^{\prime}|^2\Big)\widehat{\mathbb{G}}_{2}(\xi, \eta, \zeta, t).
\end{equation}
The estimates \eqref{4001}-\eqref{4002} and \eqref{3000} result in that
\begin{equation}\label{4003}
|\xi||\mathbb{H}_{\xi}^{\prime}(\xi, \eta, \zeta, t)|\leq C\Big(\xi^2(1+At)^2t^{2/\alpha}+\eta^2t^{2/\alpha}+\zeta^2t^{2/\alpha}\Big)^{\alpha/2}
\end{equation}
\begin{equation}\label{4004}
|\xi|^2|\mathbb{H}_{\xi}^{\prime}(\xi, \eta, \zeta, t)|^2\leq C\Big(|\xi|^2(1+At)^2t^{2/\alpha}+\eta^2t^{2/\alpha}+\zeta^2t^{2/\alpha}\Big)^{\alpha},
\end{equation}
\begin{equation}\label{4005}
|\xi|^2|\mathbb{H}_{\xi}^{\prime\prime}(\xi, \eta, \zeta, t)|\leq C\int_{0}^{t}|\xi|^2\big(\xi^2+(\eta+As\xi)^2+\zeta^2\big)^{\alpha/2-1}\, ds(1+At)^2.
\end{equation}
If $At\leq 1$, \eqref{4005} can be estimated as follows:
\begin{equation}\label{4006}
|\xi|^2|\mathbb{H}_{\xi}^{\prime\prime}(\xi, \eta, \zeta, t)|\leq C\int_{0}^{t}|\xi|^{\alpha}\, ds(1+At)^2\leq C\big(|\xi|(1+At)t^{1/\alpha}\big)^{\alpha}.
\end{equation}
If $At>1$, one has that
\begin{equation}\label{4007}
|\xi|^2|\mathbb{H}_{\xi}^{\prime\prime}(\xi, \eta, \zeta, t)|\leq C\int_{0}^{t}|\xi|^2(\eta+As\xi)^{\alpha-2}\, ds(1+At)^2\leq C\big(|\xi|(1+At)t^{1/\alpha}\big)^{\alpha}.
\end{equation}
Substituting the estimates \eqref{4003}-\eqref{4007} into \eqref{4000}, one has that
\begin{equation}\label{4008}
\Big|\partial_{\xi}^2\big(\xi^2\widehat{\mathbb{G}}_{2}(\xi, \eta, \zeta, t)\big)\Big|\leq C\Big(1+\big(|\xi|^2(1+At)^2t^{2/\alpha}+\eta^2t^{2/\alpha}+\zeta^2t^{2/\alpha}\big)^{\alpha}\Big)\widehat{\mathbb{G}}_{2}(\xi, \eta, \zeta, t),
\end{equation}
which together with the definition of $\mathbb{E}(\xi, \eta, \zeta, t)$ yields that
\begin{equation*}
\Big|\partial_{\xi}^2\big(\xi^2\widehat{\mathbb{G}}_{2}(\xi, \eta, \zeta, t)\big)\Big|\leq C\mathbb{E}(\xi, \eta, \zeta, t).
\end{equation*}
The estimate of $\Big|\partial_{\eta}^{2}\big(\xi^2\widehat{\mathbb{G}}_{2}(\xi, \eta, \zeta, t)\big)\Big|$ could be obtained similarly:
\begin{equation}\label{4009}
\Big|\partial_{\eta}^2\big(\xi^2\widehat{\mathbb{G}}_{2}(\xi, \eta, \zeta, t)\big)\Big|\leq C\Big(|\xi|^2|\mathbb{H}_{\eta}^{\prime}|^2+|\xi|^2|\mathbb{H}_{\eta}^{\prime\prime}|\Big)\widehat{\mathbb{G}}_{2}(\xi, \eta, \zeta, t).
\end{equation}
From \eqref{4010} and \eqref{3001}, one has that
\begin{equation}\label{4011}
|\xi|^2|\mathbb{H}_{\eta}^{\prime}(\xi, \eta, \zeta, t)|^2\leq C\Big(\xi^2(1+At)^2t^{2/\alpha}+\eta^2t^{2/\alpha}+\zeta^2t^{2/\alpha}\Big)^2(1+At)^{-1},
\end{equation}
\begin{equation}\label{4012}
|\xi|^2|\mathbb{H}_{\eta}^{\prime\prime}(\xi, \eta, \zeta, t)|\leq C|\xi|\int_{0}^{t}|\xi|\Big(\xi^2+(\eta+As\xi)^2+\zeta^2\Big)^{\alpha/2-1}\, ds
\leq C\Big(|\xi|(1+At)t^{1/\alpha}\Big)^{\alpha}(1+At)^{-1}.
\end{equation}
The above estimates \eqref{4011}-\eqref{4012}, the inequality \eqref{4009} and the definition of $\mathbb{E}(\xi, \eta, \zeta, t)$ result in that
\begin{equation*}
\Big|\partial_{\eta}^2\big(\xi^2\widehat{\mathbb{G}}_{2}(\xi, \eta, \zeta, t)\big)\Big|\leq C(1+At)^{-2}\mathbb{E}(\xi, \eta, \zeta, t).
\end{equation*}
The estimate for $\partial_{\zeta}^2\big(\xi^2\widehat{\mathbb{G}}_{2}(\xi, \eta, \zeta, t)$ could be obtained after a same arguement:
\begin{equation*}
\Big|\partial_{\zeta}^2\big(\xi^2\widehat{\mathbb{G}}_{2}(\xi, \eta, \zeta, t)\big)\Big|\leq C(1+At)^{-2}\mathbb{E}(\xi, \eta, \zeta, t),
\end{equation*}
and thus
\begin{equation*}
\begin{split}
& \|\partial_{x}^2\mathbb{G}_{2}(\cdot, \cdot, \cdot, t)\|_{L^1}\\
\leq & C\|\mathbb{B}(\cdot, \cdot, \cdot, t)\|_{L^2}^{1/2}\|\xi^2\widehat{\mathbb{G}}_{2}(\cdot, \cdot, \cdot, t)\|_{L^2}\\
& +C\Big(\|\mathbb{B}(\cdot, \cdot, \cdot, t)\|_{L^2}^{1/2}\|\partial_{\xi}^2(\xi^2\widehat{\mathbb{G}}_{2})(\cdot, \cdot, \cdot, t)\|_{L^2}+\|\mathbb{B}(\cdot, \cdot, \cdot, t)\|_{L^2}^{1/2}\|(\partial_{\eta}^2+\partial_{\zeta}^2)(\xi^2\widehat{\mathbb{G}}_{2})(\cdot, \cdot, \cdot, t)\|_{L^2}\Big)\\
\leq & Ct^{-2/\alpha}(1+At)^{-2}.
\end{split}
\end{equation*}\\

\textbf{Case 2. The estimates of the other five derivatives.}\\

Comparing the estimate of $\|\partial_{x}\mathbb{G}_{2}(\xi, \eta, \zeta, t)\|_{L^1}$ and the cases in $k=1$, one can conclude that $\xi^2$ could absorb an extra factor $t^{1/\alpha}(1+At)$ compared with $\xi$, and $\eta^2, \zeta^2$ can absorb an extra factor $t^{1/\alpha}$ compared with $\eta, \zeta$. Thus a corresponding modification of the estimates for $k=1$ yields that
\begin{equation*}
\|\partial_{y}^2\mathbb{G}_{2}(\cdot, \cdot, \cdot, t)\|_{L^1},\ \|\partial_{z}^2\mathbb{G}_{2}(\cdot, \cdot, \cdot, t)\|_{L^1}, \ \|\partial_{y}\partial_{z}\mathbb{G}_{2}(\cdot, \cdot, \cdot, t)\|_{L^1}\leq C t^{-2/\alpha},
\end{equation*}
and
\begin{equation*}
\|\partial_{x}\partial_{y}\mathbb{G}_{2}(\cdot, \cdot, \cdot, t)\|_{L^1},\ \|\partial_{x}\partial_{z}\mathbb{G}_{2}(\cdot, \cdot, \cdot, t)\|_{L^1}\leq Ct^{-2/\alpha}(1+At)^{-1},
\end{equation*}
and we finish the proof for $k=2$.
\end{proof}

Based on the estimate of $\mathbb{G}_{2}(x, y, z, t)$ obtained in the Lemma \ref{lem:3.6}, one could give the estimate of the Green's function $\mathbb{G}$:

\begin{lemma}\label{lem:3.7}
For $\alpha\in(1, 2],\ k=k_1+k_2+k_3=1$ or $2$ and $q\in [1, 2]$, it holds that
\begin{equation}\label{eq:3.30}
\left\|\partial_{x}^{k_{1}}\partial_{y}^{k_{2}}\partial_{z}^{k_{3}}\mathbb{G}(\cdot, \cdot, \cdot, t; x', y', z')\right\|_{L^q}\leq C t^{-\frac{3}{\alpha}\big(1-\frac{1}{q}\big)-\frac{k}{\alpha}}(1+At)^{-k_{1}-\big(1-\frac{1}{q}\big)}
\end{equation}
and
\begin{equation}\label{eq:3.31}
\left\|\partial_{x}^{k_{1}}\partial_{y}^{k_{2}}\partial_{z}^{k_{3}}\mathbb{G}(x, y, z, t; \cdot, \cdot, \cdot)\right\|_{L^q}\leq C t^{-\frac{3}{\alpha}\big(1-\frac{1}{q}\big)-\frac{k}{\alpha}}(1+At)^{-k_{1}-\big(1-\frac{1}{q}\big)},
\end{equation}
for a positive constant $C$.
\end{lemma}

\begin{proof}
The proof for $q=2$ is similar to the one for Lemma \ref{lem:3.5}:
$$
\left\|\partial_{x}^{k_{1}}\partial_{y}^{k_2}\partial_{z}^{k_{3}} \mathbb{G}_{2}(\cdot, \cdot, \cdot, t)\right\|_{L^2}\leq Ct^{-\frac{3}{2\alpha}-\frac{k}{\alpha}}(1+At)^{-k_{1}-\frac{1}{2}},
$$
which together with Lemma \ref{lem:3.6} and  the interpolation inequality yields that
\begin{equation}\label{3003}
\|\partial_{x}^{k_{1}}\partial_{y}^{k_2}\partial_{z}^{k_{3}} \mathbb{G}_{2}(\cdot, \cdot, \cdot, t)\|_{L^q}\leq Ct^{-\frac{3}{\alpha}\big(1-\frac{1}{q}\big)-\frac{k}{\alpha}}(1+At)^{-k_{1}-\big(1-\frac{1}{q}\big)}.
\end{equation}
It combined with the definition of $\widehat{\mathbb{G}}_{1}(\xi, \eta, \zeta, t; x', y', z')$ verifies \eqref{3003} and then \eqref{eq:3.31} could be verified by \eqref{eq:3.30} and \eqref{eq:3.9}.
\end{proof}

\subsection{Conclusion}

The Green's function is a kernel function satisfying
$$
\mathbb{G}(x, y, z, t; x', y', z')=\mathbb{G}(x-x', y, z-z', t; y')=F(x-x'-Aty', y-y', z-z', t),
$$
and thus its derivatives satisfy the following identities:
\begin{equation}\label{eq:3.32}
\begin{aligned}
\partial_{x}\mathbb{G}(x-x', y, z-z', t; y')=&-\partial_{x'}\mathbb{G}(x-x', y, z-z', t; y'),\\
\partial_{y}\mathbb{G}(x-x', y, z-z', t; y')=&-\partial_{y'}\mathbb{G}(x-x', y, z-z', t; y')\\
&+At\partial_{x'}\mathbb{G}(x-x', y, z-z', t; y'),\\
\partial_{z}\mathbb{G}(x-x', y, z-z', t; y')=&-\partial_{z'}\mathbb{G}(x-x', y, z-z', t; y').
\end{aligned}
\end{equation}

\begin{remark}\label{rem:2}
The above identities \eqref{eq:3.32} help to apply integration by parts in the closure of nonlinearity.
\end{remark}

For convenience of presentation, denote
$$
\mathbb{G}(t)\circledast f=\int_{\mathbb{R}^3}\mathbb{G}(x-x', y, z-z', t; y')f(x', y', z')\, dx'\, dy'\, dz',
$$
and
\begin{equation}\label{eq:3.33}
\interleave\mathbb{G}(t)\interleave_{L^p}=\max\Big\{\|\mathbb{G}(\cdot-x', \cdot, \cdot-z', t; y')\|_{L^p}, \|\mathbb{G}(x-\cdot, y, z-\cdot, t; \cdot)\|_{L^p}\Big\}.
\end{equation}

The following is the main result of this section:

\begin{proposition}\label{prop:3.8}
For any non-negative integer $k=k_{1}+k_{2}+k_{3},\ p\geq 2$, it holds that
\begin{equation}\label{eq:3.34}
\begin{aligned}
\interleave\partial_{x}^{k_{1}}\partial_{y}^{k_{2}}\partial_{z}^{k_{3}}\mathbb{G}(t)\interleave_{L^p}
\leq&Ct^{-\frac{3}{\alpha}\big(1-\frac{1}{p}\big)-\frac{k}{\alpha}}(1+At)^{-k_{1}-\big(1-\frac{1}{p}\big)},\\
\interleave\partial_{x'}^{k_{1}}\partial_{y'}^{k_{2}}\partial_{z'}^{k_{3}}\mathbb{G}(t)\interleave_{L^p}
\leq&Ct^{-\frac{3}{\alpha}\big(1-\frac{1}{p}\big)-\frac{k}{\alpha}}(1+At)^{-k_{1}-\big(1-\frac{1}{p}\big)}.
\end{aligned}
\end{equation}
Moreover, if $k=1, 2$, for $\ q\geq 1$, it holds that
\begin{equation}\label{eq:3.35}
\begin{aligned}
\interleave\partial_{x}^{k_{1}}\partial_{y}^{k_{2}}\partial_{z}^{k_{3}}\mathbb{G}(t)\interleave_{L^q}
\leq&Ct^{-\frac{3}{\alpha}\big(1-\frac{1}{q}\big)-\frac{k}{\alpha}}(1+At)^{-k_{1}-\big(1-\frac{1}{q}\big)},\\
\interleave\partial_{x'}^{k_{1}}\partial_{y'}^{k_{2}}\partial_{z'}^{k_{3}}\mathbb{G}(t)\interleave_{L^q}
\leq& Ct^{-\frac{3}{\alpha}\big(1-\frac{1}{q}\big)-\frac{k}{\alpha}}(1+At)^{-k_{1}-\big(1-\frac{1}{q}\big)}.
\end{aligned}
\end{equation}

If $\Lambda$ is a quasi-differential operator with the symbol $\sigma(\Lambda)=\sqrt{\xi^2+\eta^2+\zeta^2}$, and $\widetilde{\beta}\geq 0, p\ge 2$, then
\begin{equation}\label{eq:3.36}
\begin{aligned}
\interleave\Lambda_{x, y, z}^{\widetilde{\beta}}\mathbb{G}(t)\interleave_{L^p}\leq&Ct^{-\frac{3}{\alpha}\big(1-\frac{1}{p}\big)-\frac{\widetilde{\beta}}{\alpha}}(1+At)^{-\big(1-\frac{1}{p}\big)},\\
\interleave\Lambda_{x^\prime, y^\prime, z^\prime}^{\widetilde{\beta}}\mathbb{G}(t)\interleave_{L^p}\leq&Ct^{-\frac{3}{\alpha}\big(1-\frac{1}{p}\big)-\frac{\widetilde{\beta}}{\alpha}}(1+At)^{-\big(1-\frac{1}{p}\big)}.
\end{aligned}
\end{equation}
Furthermore, for any $1<\widetilde{\gamma}\leq 2$ and $q\ge 1$, one has that
\begin{equation}\label{eq:3006}
\interleave\Lambda_{x, y, z}^{\widetilde{\gamma}}\mathbb{G}(t)\interleave_{L^q}\leq Ct^{-\frac{3}{\alpha}\big(1-\frac{1}{p}\big)-\frac{\widetilde{\gamma}}{\alpha}}(1+At)^{-\big(1-\frac{1}{p}\big)}.
\end{equation}
\end{proposition}

\begin{proof}
From Lemma \ref{lem:3.5}, one has that
\begin{equation}\label{eq:3.38}
\begin{aligned}
\|\partial_{x}^{k_{1}}\partial_{y}^{k_{2}}\partial_{z}^{k_{3}}\mathbb{G}(\cdot-x', \cdot, \cdot-z', t; y')\|_{L^p}\leq C t^{-\frac{3}{\alpha}\big(1-\frac{1}{p}\big)-\frac{k}{\alpha}}(1+At)^{-k_{1}-\big(1-\frac{1}{p}\big)},\\
\|\partial_{x}^{k_{1}}\partial_{y}^{k_{2}}\partial_{z}^{k_{3}}\mathbb{G}(x-\cdot, y, z-\cdot, t; \cdot)\|_{L^p}\leq Ct^{-\frac{3}{\alpha}\big(1-\frac{1}{p}\big)-\frac{k}{\alpha}}(1+At)^{-k_{1}-\big(1-\frac{1}{p}\big)},
\end{aligned}
\end{equation}
which together with the definition of $\interleave\mathbb{G}(t)\interleave_{L^p}$ verifies $\eqref{eq:3.34}_{1}$ and then the estimate $\eqref{eq:3.34}_{1}$ combined with \eqref{eq:3.32} yields $\eqref{eq:3.34}_{2}$.

One could repeat the above argument and replace the estimates \eqref{eq:3.38} by the ones in Lemma \ref{lem:3.7} to prove \eqref{eq:3.35}. After that, the inequalities \eqref{eq:3.36} and \eqref{eq:3006} could be easily obtained by Gagliardo-Nirenberg inequality (see \cite{BH.2001, EA.2015}) for fractional derivatives and \eqref{eq:3.34}-\eqref{eq:3.35}.
\end{proof}

\section{The global existence and large time behavior of the solution}\label{sec.4}

In this section, we first give a local existence of the solution and also the blow-up criterion. Later, an $L^p\ (2\leq p<\infty)$ bound of the solution results in the global existence and the $L^p\ (2\leq p<\infty)$ bound could be obtained by the Green's function $\mathbb{G}(x, y, z, t; x', y', z')$ constructed in last section.

\vskip .05in

First, the local well-posedness and the $L^p (p\geq2)$-criterion of solution to Equation \eqref{eq:1.1} are established. It is as follows.

\begin{theorem}\label{thm:4.1}
Let $\alpha\in (1,2]$ and some $p\in [2, \infty)$, if initial data $n_{0}(x, y, z)\geq 0$ and $\ n_{0}(x, y, z)\in W^{3, p}(\mathbb{R}^3)\cap L^{1}(\mathbb{R}^3)$, there exists a positive constant $T_{0}=T(n_{0},\alpha)$, such that when $t\leq T_{0}$, there exists a unique non-negative classical solution $n$ for the system \eqref{eq:1.1} satisfying
$$
n(x, y, z, t)\in C\big(W^{3, p}(\mathbb{R}^3)\cap L^1(\mathbb{R}^{3}); [0, T_{0}]\big).
$$
Moreover, for any given time $T_{1}$, if the solution satisfies
\begin{equation}\label{eq:4.1}
\lim_{t\to T_{1}}\sup_{0\leq s\leq T_{1}}\|n(\cdot, \cdot,\cdot, s)\|_{L^{p}(\mathbb{R}^3)}<\infty,
\end{equation}
then there exists a sufficiently small constant $\varepsilon>0$ such that the solution could be extended from the time $T_{1}$ to $T_{1}+\varepsilon$.
\end{theorem}

The non-negative of the solution $n(x, y, z, t)$ of the equation \eqref{eq:1.1}  has already been proved in many references and we omit the details here. The local existence of the solution to the system \eqref{eq:1.1} could be also proved by the standard method. And the Theorem \ref{thm:4.1} tells us that we can obtain the global classical solution if the $L^p$ estimate of solution to equation \eqref{eq:1.1} is global. In this paper, we will obtain the global  estimates of higher-order derivatives of solution based on global $L^p$ decay, see Lemma \ref{lem:4.6}.

\vskip .05in

Based on the local existence of the solution, the Bootstrap argument could be used to extend the local solution to a global solution: we give the following Bootstrap hypotheses with $\delta$ to be determined later:

\begin{equation}\label{eq:4.2}
\|n(\cdot, \cdot, \cdot, t)\|_{L^{p}}\leq 2\delta(1+t)^{-\big(\frac{3}{\alpha}+1\big){\big(1-\frac 1 p\big)}}\ \ \ \ \text{ for }2\le p< \infty.
\end{equation}
Combining \eqref{eq:4.2} and the conservation of mass about $n(x, y, z, t)$, one has that
\begin{equation}\label{eq:4.3}
\|n(\cdot, \cdot, \cdot, t)\|_{L^q}\leq C\delta(1+t)^{-\big(\frac{3}{\alpha}+1\big)\big(1-\frac 1 q\big)}\ \ \ \ \text{ for }1\leq q< \infty.
\end{equation}
By the typical steps in Bootstrap argument, we are going to show that the constant $2\delta$ in \eqref{eq:4.2} could be replaced by $\delta$.

\begin{lemma}\label{lem:4.2}
If the solution $n(x,y,z,t)$ of \eqref{eq:1.1} satisfies \eqref{eq:4.2} with the constant $\delta$ given by
\begin{equation}\label{eq:4.4}
\delta=80\left(1+C_{2}\right)\max\Big\{\|n_{0}(\cdot, \cdot, \cdot)\|_{L^1}, \|n_{0}(\cdot, \cdot, \cdot)\|_{L^{\infty}}\Big\},
\end{equation}
then for any $2\le p<\infty$ it also satisfies
\begin{equation}\label{eq:4.5}
\|n(\cdot, \cdot, \cdot, t)\|_{L^{p}}\leq\delta(1+t)^{-\big(\frac{3}{\alpha}+1\big)\big(1-\frac{1}{p}\big)},
\end{equation}
where $C_{2}$ is given later.
\end{lemma}

Before giving the proof of Lemma \ref{lem:4.2}, we list some estimates about the nonlinear term $\mathbf{B}(n)$.

\begin{lemma}\label{lem:4.3}
Under the assumption \eqref{eq:4.3}, for any $3/2\leq  r< \infty$, it holds that
\begin{equation}\label{eq:4.6}
\|\mathbf{B}(n)(\cdot, \cdot, \cdot, t)\|_{L^r}\leq C\delta(1+t)^{-\big(\frac{3}{\alpha}+1\big)\big(\frac{2}{3}-\frac{1}{r}\big)}.
\end{equation}
Moreover for $\widetilde{r}>0$, we have
\begin{equation}\label{eq:4.7}
\|\Lambda^{\widetilde{r}}\mathbf{B}(n)(\cdot, \cdot, \cdot, t)\|_{L^r}\leq C\|\Lambda^{\widetilde{r}} n(\cdot, \cdot, \cdot, t)\|_{L^{\frac{3r}{3+r}}}.
\end{equation}
\end{lemma}

\begin{proof}
We rewrite $\mathbf{B}(n)$ as
$$
\mathbf{B}(n)=\nabla\big((-\Delta)^{-1} n\big)=\widetilde{K}(D) n.
$$
The symbol of $\widetilde{K}(D)$ is a homogeneous function of degree $-1$. According to the Hardy-Littlewood-Sobolev theorem, the operator $\widetilde{K}(D)$ satisfies
$$
\|\mathbf{B}(n)(\cdot, \cdot, \cdot, t)\|_{L^r}\leq C\|\big(\Lambda\widetilde{K}(D)\big)n(\cdot, \cdot, \cdot, t)\|_{L^{\frac{3r}{3+r}}}.
$$
Here, $\Lambda\widetilde{K}(D)$ is a zero-order operator and it is a bounded operator in $L^{\frac{3r}{3+r}}(\mathbb{R}^3)$. Thus one could combine the estimates of $\|n(\cdot, \cdot, \cdot, t)\|_{L^q}$ to prove \eqref{eq:4.6}. For (\ref{eq:4.7}), it could be proved by the fact that
$$
\Lambda^{\widetilde{r}}\mathbf{B}(n)=\left(\Lambda^{\widetilde{r}}\widetilde{K}(D)\Lambda^{-\widetilde{r}}\right)\Lambda^{\widetilde{r}} n,
$$
where the symbol of $\Lambda^{\widetilde{r}}\widetilde{K}(D)\Lambda^{-\widetilde{r}}$  is also a homogeneous function of degree $-1$ and a similar argument in the proof of (\ref{eq:4.6}).
\end{proof}

The Duhamel principle and the Green's function result in the following representation of the solution $n(x,y,z,t)$ of \eqref{eq:1.1}:
\begin{equation}\label{eq:4.8}
\begin{aligned}
n(x, y, z, t)=&\int_{\mathbb{R}^3}\mathbb{G}(x, y, z, t; x',y',z')n_{0}(x', y', z')\, dx'\, dy'\, dz'\\
&+\int_{0}^{t}\int_{\mathbb{R}^3}\mathbb{G}(x, y, z, t-s; x',y',z')\nabla_{x', y', z'}\cdot\big(n \mathbf{B}(n)\big)(x', y', z', s)\, dx'\, dy'\, dz'\, ds\\
\equiv & \mathbb{G}(t)\circledast n_0+\int_0^t \mathbb{G}(t-s)\circledast \nabla_{x', y', z'}\cdot\left(n \mathbf{B}(n)\right)(s)\, ds.
\end{aligned}
\end{equation}

The equation \eqref{eq:4.8} shows that the solution $n(x, y, z, t)$ includes two parts: initial propagation and the nonlinear term. In order to get the $L^p\ (2\leq p<\infty)$ estimates of the solution $n(x, y, z, t)$, we need to obtain them respectively. Lemma \ref{lem:3.7n} is very important to obtain the boundedness of the initial propagation: Based on the Lemma \ref{lem:3.7n}, we give the estimate of the boundedness about the initial data term.

\begin{lemma}\label{lem:4.51}
For the initial data term $\mathbb{G}(t)\circledast n_{0}$ and any $p\in [2, \infty)$, we have
\begin{equation*}
\|\mathbb{G}(t)\circledast n_{0}\|_{L^p}\leq C.
\end{equation*}
\end{lemma}

\begin{proof}
Let us denote $u(x,y,z,t)=\mathbb{G}(t)\circledast n_{0}=\int_{\mathbb{R}^3} \mathbb{G}(x, y, z, t; x', y', z')n_{0}(x',y',z')\, dx'\, dy'\, dz'$. Then we know that $u(x, y, z, t)$ satisfy the following equation:
\begin{equation}\label{eq:4.8n}
\begin{cases}
\partial_{t}\hat{u}-A\xi\partial_{\eta}\hat{u}+(\xi^2+\eta^2+\zeta^2)^{\alpha/2}\hat{u}=0,\\
\hat{u}(\xi, \eta,\zeta,0)=\hat{n}_{0}(\xi, \eta,\zeta).
\end{cases}
\end{equation}
The equation (\ref{eq:4.8n}) can be represented as
\begin{equation*}
\hat{u}(\xi, \eta,\zeta,t)=\exp\Bigg\{-\int_{0}^{t}\big(\xi^2+(\eta+As \xi)^2+\zeta^2\big)^{\alpha/2}\, ds\Bigg\}\hat{n}_{0}(\xi, \eta+At \xi, \zeta).
\end{equation*}
Using the definition of Fourier transform, we have
\begin{equation}\label{eq:4.9n}
\begin{split}
|u(x,y,z,t)|=&\Bigg|\int_{\mathbb{R}^3}e^{i(x\xi+y\eta+z\zeta)}\hat{u}(\xi,\eta,\zeta,t)\, d\xi d\eta d\zeta\Bigg|\\
\leq&C\int_{\mathbb{R}^3}|\hat{u}(\xi,\eta,\zeta,t)|\, d\xi d\eta d\zeta\leq C\int_{\mathbb{R}^3}|\hat{n}_{0}(\xi,\eta+At \xi,\zeta)|\, d\xi d\eta d\zeta.
\end{split}
\end{equation}
As $At \xi $ is a shift term, we have
\begin{equation*}
\int_{\mathbb{R}^3}|\hat{n}_{0}(\xi,\eta+At \xi, \zeta)|\, d\xi d\eta d\zeta=\int_{\mathbb{R}^3}|\hat{n}_{0}(\xi,\eta,\zeta)|\, d\xi d\eta d\zeta.
\end{equation*}
Combining the equation (\ref{eq:4.9n}) and Lemma \ref{lem:3.7n}, we have
\begin{equation*}
\|u(\cdot, \cdot, \cdot, t)\|_{L^{\infty}}\leq C\|\hat{n}_{0}(\cdot, \cdot, \cdot, t)\|_{L^1}\leq C\|n_{0}(\cdot, \cdot, \cdot, t)\|_{H^2}\leq C.
\end{equation*}
For the $\|u(\cdot, \cdot, \cdot, t)\|_{L^2}$, using Parseval formula and the inequality \eqref{eq:4.9n}, we have
\begin{equation*}
\|u(\cdot, \cdot, \cdot, t)\|_{L^2}\leq\|\hat{u}(\cdot, \cdot, \cdot, t)\|_{L^2}\leq\|\hat{n}_{0}(\cdot, \cdot+At\cdot, \cdot)\|_{L^2}\leq\|\hat{n}_{0}(\cdot, \cdot, \cdot)\|_{L^2}\leq\|n_{0}(\cdot, \cdot, \cdot)\|_{L^2}\leq C.
\end{equation*}
Then using the interpolation inequality, for any $p\in[2,\infty)$, we have
\begin{equation*}
\|u(\cdot, \cdot, \cdot, t)\|_{L^p}\leq C.
\end{equation*}
\end{proof}

Now we give the proof of Lemma \ref{lem:4.2}.

\begin{proof}[Proof of Lemma \ref{lem:4.2}.]
First we consider the initial propagation. For $2\le p<\infty$, Young's inequality in Lemma \ref{lem:2.1} and the $L^p$ estimates for Green's function in Proposition \ref{prop:3.8} result in that
\begin{equation}\label{eq:4.9}
\left\|\mathbb{G}(t)\circledast n_0\right\|_{L^p}
\leq C_{1}\interleave\mathbb{G}(t)\interleave_{L^{p}}\left\|n_{0}(\cdot, \cdot, \cdot)\right\|_{L^1}\leq C_{2}\left\|n_{0}(\cdot, \cdot, \cdot)\right\|_{L^1}\cdot
t^{-\frac{3}{\alpha}\big(1-\frac 1p\big)}(1+At)^{-1-\frac 1p}.
\end{equation}
Here the constant $C_{2}$ is determined by Proposition \ref{prop:3.8} and the notation $\interleave\mathbb{G}(t)\interleave_{L^{p}}$ is defined by \eqref{eq:3.33}. On the other hand, by Lemma \ref{lem:4.51}, we have
\begin{equation}\label{eq:4.10}
\left\|\mathbb{G}(t)\circledast n_0\right\|_{L^p}\leq C.
\end{equation}
If $A>1$ and $\delta$ is given by \eqref{eq:4.4}, combining \eqref{eq:4.9} and \eqref{eq:4.10}, one has that
\begin{equation}\label{eq:4.11}
\left\|\mathbb{G}(t)\circledast n_0\right\|_{L^p}\leq\frac{\delta}{10}(1+t)^{-\big(\frac{3}{\alpha}+1\big)\big(1-\frac{1}{p}\big)}\ \ \ \ \text{ for }2\le p<\infty.
\end{equation}

Then we consider the nonlinear term. We deal with it in different time scales.

\vskip .05in

\textbf{Case 1:} $t\ge 1$. Then $t\ge 2\cdot A^{-\theta}$ for $A\gg 1$ and any positive constant $\theta$ which implies $t-A^{-\theta}\ge t/2$. If $2\leq p<\infty$, by integrating by part, H\"older's inequality, Proposition \ref{prop:3.8}, Lemma \ref{lem:4.3} and the assumption \eqref{eq:4.2}, we have
\begin{equation}\label{eq:4.12}
\begin{aligned}
&\left\|\int_{t-A^{-\theta}}^t\mathbb{G}(t-s)\circledast\nabla_{x', y', z'}\cdot\big(n \mathbf{B}(n)\big)(s)\, ds\right\|_{L^p}\\
=& \left\|\int_{t-A^{-\theta}}^t \nabla_{x',y',z'}\mathbb{G}(t-s)\circledast \big(n \mathbf{B}(n)\big)(s)\, ds\right\|_{L^p}\\
\leq & C\int_{t-A^{-\theta}}^{t}\interleave\nabla_{x',y',z'}\mathbb{G}(t-s)\interleave_{L^{1}} \|\big(n\mathbf{B}(n)\big)(s)\|_{L^p}\, ds\\
\leq & C\int_{t-A^{-\theta}}^{t}(t-s)^{-1/\alpha}\big\|n(s)\big\|_{L^{2p}}\|\mathbf{B}(n)(s)\|_{L^{2p}}\, ds\\
\leq & C\delta^2(1+t)^{-\big(\frac{3}{\alpha}+1\big)\big(\frac{5}{3}-\frac 1p\big)}\int_{t-A^{-\theta}}^{t}(t-s)^{-1/\alpha}\, ds\\
\leq & C\delta^2 A^{-\frac{(\alpha-1)\theta}{\alpha}}(1+t)^{-\big(\frac{3}{\alpha}+1\big)\big(1-\frac 1p\big)}.
\end{aligned}
\end{equation}
In \eqref{eq:4.12}, since $\alpha>1$, for A sufficiently large, one has that
\begin{equation}\label{eq:4.13}
\left\|\int_{t-A^{-\theta}}^t\mathbb{G}(t-s)\circledast\nabla_{x', y', z'}\big(n\mathbf{B}(n)\big)(s)ds\right\|_{L^p}
\leq\frac{\delta}{10}(1+t)^{-\big(\frac{3}{\alpha}+1\big)\big(1-\frac 1p\big)}.
\end{equation}
In the time interval $[t/2,t-A^{-\theta}]$, taking the $L^2$ norm of the Green's function $\nabla_{x', y', z'}\mathbb{G}$ and combining Proposition \ref{prop:3.8}, Young's inequality for kernel type in Lemma \ref{lem:2.1} and H\"older's inequality, one has that
\begin{equation}\label{eq:4.14}
\begin{aligned}
&\left\|\int^{t-A^{-\theta}}_{t/2}\mathbb{G}(t-s)\circledast\nabla_{x', y', z'}\cdot\big(n \mathbf{B}(n)\big)(s)\, ds\right\|_{L^p}\\
\leq & C\int^{t-A^{-\theta}}_{t/2}\interleave\nabla_{x',y',z'}\mathbb{G}(t-s)\interleave_{L^2}\left\|\big(n \mathbf{B}(n)\big)(s)\right\|_{L^{\frac{2p}{p+2}}}\, ds\\
\leq& C\int_{t/2}^{t-A^{-\theta}}(t-s)^{-5/(2\alpha)}\big(1+A(t-s)\big)^{-1/2}\cdot
\left\|n(s)\right\|_{L^p}\left\|\mathbf{B}(n)(s)\right\|_{L^2}\, ds\\
\leq& C\delta^2(1+t)^{-\big(\frac{3}{\alpha}+1\big)\big(1-\frac 1p\big)}\int_{t/2}^{t-A^{-\theta}}A^{-\frac 12}(t-s)^{-\frac{5}{2\alpha}-\frac{1}{2}}ds\\
\le & C\delta^2 A^{-\frac 12} A^{\big(\frac 5{2\alpha}-\frac 12\big)\theta} (1+t)^{-\big(\frac{3}{\alpha}+1\big)\big(1-\frac 1p\big)}\leq \frac \delta {10} (1+t)^{-\big(\frac{3}{\alpha}+1\big)\big(1-\frac 1p\big)},
\end{aligned}
\end{equation}
for $A$ sufficiently large. Here, we require
\begin{equation}\label{eq:4.15}
0<\theta<\alpha/(5-\alpha).
\end{equation}
to ensure $-1/2+\big(5/(2\alpha)-1/2\big)\theta<0$. In the time interval $[0,t/2]$, it holds that
\begin{equation}\label{eq:4.16}
\begin{aligned}
&\left\|\int_0^{t/2}\mathbb{G}(t-s)\circledast\nabla_{x', y', z'}\cdot\big(n \mathbf{B}(n)\big)(s)\, ds\right\|_{L^p}\\
\leq & C\int_0^{t/2}\interleave\nabla_{x',y',z'}\mathbb{G}(t-s)\interleave_{L^{p}}\left\|\big(n \mathbf{B}(n)\big)(s)\right\|_{L^{1}}\, ds\\
\leq&C\int_0^{t/2}(t-s)^{-\frac{3}{\alpha}\big(1-\frac{1}{p}\big)-\frac{1}{\alpha}}\big(1+A(t-s)\big)^{-1-\frac 1p}
\left\|n(s)\right\|_{L^{2}}\left\|\mathbf{B}(n)(s)\right\|_{L^{2}}\, ds\\
\leq&C\delta^2 A^{-1-\frac 1p}t^{-\big(\frac{3}{\alpha}+1\big)\big(1-\frac{1}{p}\big)-\frac{1}{\alpha}}
\int_0^{t/2}(1+s)^{-\frac 3{2\alpha}-\frac 12}\, ds\\
\leq & C\delta^2 A^{-1-\frac 1p}(1+t)^{-\big(\frac{3}{\alpha}+1\big)\big(1-\frac 1p\big)}\le\frac\delta {10} (1+t)^{-\big(\frac{3}{\alpha}+1\big)\big(1-\frac{1}{p}\big)}
\end{aligned}
\end{equation}
since $A\gg 1$, $t\ge 1$ and $-3/(2\alpha)-1/2\leq -1$ for $\alpha\in (1,2]$.

\vskip .05in

\textbf{Case 2:} $2 A^{-\theta}<t<1$ for a positive constant $\theta$ satisfying \eqref{eq:4.15}. It holds that
\begin{equation}\label{eq:4.17}
\begin{split}
&\left\|\int_{t/2}^t\mathbb{G}(t-s)\circledast\nabla_{x', y', z'}\cdot\big(n \mathbf{B}(n)\big)(s)\, ds\right\|_{L^p}\\
\leq & C\int_{t/2}^{t}\interleave\nabla_{x', y', z'}\mathbb{G}(t-s)\interleave_{L^{\frac{7+\alpha}{8}}}\|n(s)\|_{L^{p}}\|\mathbf{B}(s)\|_{L^{\frac{7+\alpha}{\alpha-1}}}\, ds\\
\leq & C\delta^2\int_{t/2}^{t}(t-s)^{-\frac{3}{\alpha}\big(1-\frac{8}{7+\alpha}\big)-\frac{1}{\alpha}}\big(1+A(t-s)\big)^{-\big(1-\frac{8}{7+\alpha}\big)}(1+s)^{-\big(\frac{3}{\alpha}+1\big)\big(1-\frac{1}{p}\big)}\, ds\\
\leq & C\delta^2A^{-\frac{\alpha-1}{7+\alpha}}(1+t)^{-\big(\frac{3}{\alpha}+1\big)\big(1-\frac 1p\big)}\int_{t/2}^{t}(t-s)^{{-\frac{3}{\alpha}\big(1-\frac{8}{7+\alpha}\big)-\frac{1}{\alpha}}}\, ds\leq\frac{\delta}{10}(1+t)^{-\big(\frac{3}{\alpha}+1\big)\big(1-\frac 1p\big)}.
\end{split}
\end{equation}
We split the remaining time interval $[0,t/2]$ into $[0,A^{-\theta}]$ and $[A^{-\theta},t/2]$. Similar to \eqref{eq:4.12}, one has that
\begin{equation}\label{eq:4.18}
\begin{aligned}
&\left\|\int_0^{A^{-\theta}}\nabla_{x',y',z'}\mathbb{G}(t-s)\circledast \big(n \mathbf{B}(n)\big)(s)ds\right\|_{L^p}\\
\leq & C\int_{0}^{A^{-\theta}} (t-s)^{-1/\alpha}\cdot(\|n(s)\|_{L^{2p}}\|\mathbf{B}(n) (s)\|_{L^{2p}})\, ds\\
\leq& C\delta^2A^{-\theta\big(1-\frac 1\alpha\big)}\leq\frac{\delta}{10}(1+t)^{-\big(\frac{3}{\alpha}+1\big)\big(1-\frac{1}{p}\big)}.
\end{aligned}
\end{equation}

In the time interval $[A^{-\theta},t/2]$, similar to \eqref{eq:4.14}, it holds that
\begin{equation}\label{eq:4.19}
\begin{aligned}
&\left\|\int_{A^{-\theta}}^{t/2}\mathbb{G}(t-s)\circledast\nabla_{x', y', z'}\cdot\big(n \mathbf{B}(n)\big)(s)\, ds\right\|_{L^p}\\
\leq & C\int_{A^{-\theta}}^{t/2}\interleave\nabla_{x',y',z'}\mathbb{G}(t-s)\interleave_{L^2}\left\|\big(n \mathbf{B}(n)\big)(s)\right\|_{L^{\frac{2p}{p+2}}}\, ds\\
\leq & C\delta^2\int^{t/2}_{A^{-\theta}}(t-s)^{-\frac{3}{2\alpha}-\frac{1}{\alpha}}\big(1+A(t-s)\big)^{-\frac{1}{2}}\, ds\\
\leq & C\delta^2\int^{t/2}_{A^{-\theta}}A^{-\frac 12}(t-s)^{-\frac{5}{2\alpha}-\frac{1}{2}}\, ds\leq\frac \delta {10}
\end{aligned}
\end{equation}
since $A\gg1$ and $-1/2+\left(5/(2\alpha)-1/2\right)\theta<0$ ensured by \eqref{eq:4.15}. Here, one also uses the fact that $-5/(2\alpha)+1/2<0$ for $\alpha\in(1,2]$.

\vskip .05in

\textbf{Case 3:} $t\leq 2\cdot A^{-\theta}\ll 1$. Similar to \eqref{eq:4.14} and \eqref{eq:4.18}, it holds that
\begin{equation}\label{eq:4.20}
\begin{aligned}
&\left\|\int_{t/2}^t \nabla_{x',y',z'}\mathbb{G}(t-s)\circledast\big(n \mathbf{B}(n)\big)(s) ds\right\|_{L^p}\\
\leq&C\int_{t/2}^{t}(t-s)^{-\frac{1}{\alpha}}\cdot\|n (s)\|_{L^{2p}}\|\mathbf{B}(n) (s)\|_{L^{2p}}\, ds\\
\leq& C\delta^2 A^{-\big(1-\frac 1\alpha\big)\theta}\leq\frac \delta {10},
\end{aligned}
\end{equation}
since $A\gg 1$ and $1-1/\alpha>0$ for $\alpha>1$. In the time interval $[0,t/2]$,  one has that
\begin{equation}\label{eq:4.21}
\begin{aligned}
&\left\|\int_0^{t/2}\nabla_{x',y',z'}\mathbb{G}(t-s)\circledast\big(n \mathbf{B}(n)\big)(s)ds\right\|_{L^p}\\
\leq&C\int_{0}^{t/2} (t-s)^{-\frac 1\alpha}\cdot\|n (s)\|_{L^{2p}}\|\mathbf{B}(n) (s)\|_{L^{2p}}\, ds\\
\leq & C\delta^2 A^{-\big(1-\frac{1}{\alpha}\big)\theta}\leq\frac{\delta}{10},
\end{aligned}
\end{equation}
for $A\gg 1$ since $1-1/\alpha>0$ for $\alpha>1$.

\vskip .05in

Combining \eqref{eq:4.8} and the estimates \eqref{eq:4.13}-\eqref{eq:4.21}, we verify \eqref{eq:4.5} and finish the proof of Lemma \ref{lem:4.2}.
\end{proof}

We have obtained the regularity criterion and the decaying rates in $L^p$ norms. Next, we go on with the estimates of derivatives under \eqref{eq:4.2}. In the following discussion, the constants
$\delta$ and $A$ which are very important for the proof can be deal with the usual constant.

\vskip .05in

Denote $\Lambda(D)$ the first-order quasi-differential operator with the symbol $\sigma\left(\Lambda(D)\right)=|\Xi|$. Then $\sigma(\Lambda^{k\gamma}(D))=|\Xi|^{k\gamma}$ .In our case, as the variable $y$ is different with the variables $x$ and $z$, we also denote $\widetilde{\Lambda}^{k\gamma}=\sum_{k_{1}+k_{2}+k_{3}=k}\Lambda^{k_{1}\gamma}\Lambda^{k_{2}\gamma}\Lambda^{k_{3}\gamma}$ and $k, k_{1}, k_{2}, k_{3}\in \mathbb{N}^{+}$. It is easy to know that the operator $\Lambda^{k\gamma}$ is equal to $\widetilde{\Lambda}^{k\gamma}$. Using the inequality \eqref{eq:3.32}, we know that
\begin{equation*}
\begin{split}
\interleave\Lambda_{x', y', z'}^{k\gamma}\mathbb{G}(t)\interleave_{L^p}=\interleave\Lambda_{x'}^{k_{1}\gamma}\Lambda_{y'}^{k_{2}\gamma}\Lambda_{z'}^{k_{3}\gamma}\mathbb{G}(t)\interleave_{L^p}\leq & \interleave\Lambda_{x'}^{k_{1}\gamma}\big(\Lambda_{y'}^{k_{2}\gamma}+(At)^{k_{2}\gamma}\Lambda_{x'}^{k_{2}\gamma}\big)\Lambda_{z'}^{k_{3}\gamma}\mathbb{G}(t)\interleave_{L^p}\\
\leq & \interleave\Lambda_{x'}^{k_{1}\gamma}\Lambda_{y'}^{k_{2}\gamma}\Lambda_{z'}^{k_{3}\gamma}\mathbb{G}(t)\interleave_{L^p}+(At)^{k_{2}\gamma}\interleave\Lambda_{x'}^{(k_{1}+k_{2})\gamma}\Lambda_{z'}^{k_{3}\gamma}\mathbb{G}(t)\interleave_{L^p}\\
\leq & Ct^{-\frac{3}{\alpha}\big(1-\frac{1}{p}\big)}(1+At)^{-k_{1}\gamma-\big(1-\frac{1}{p}\big)}.
\end{split}
\end{equation*}

{From the above inequality, we know that if we change the variable from $x, y, z$ to $x', y', z'$, it does not impact the decay rate.}

\begin{lemma}\label{lem:4.5}
With the initial function $n_{0}(x,y,z)\in W^{3, p}(\mathbb{R}^3)\cap L^1(\mathbb{R}^3)$ and for
\begin{equation}\label{eq:4.22}
0<\gamma< \alpha-1, \ \ \ \ k\gamma\le 3,\ k\in\mathbb{N}^+,\end{equation}
one has that
$$
\left\|\Lambda^{k\gamma}\left(\mathbb{G}(t)\circledast n_{0}\right)\right\|_{L^{p}}\leq C(1+t)^{-\big(\frac{3}{\alpha}+1\big)\big(1-\frac 1p\big)-\frac{k\gamma}{\alpha}}\ \ \ \ \text{ for } 2\le p<\infty.
$$
\end{lemma}

\begin{proof}
From Lemma \ref{lem:2.1}, for any $k\ge 0$ one has that
\begin{equation}\label{eq:4.23}
\left\|\Lambda^{k\gamma}(\mathbb{G}(t)\circledast n_{0})\right\|_{L^{p}}\leq C\interleave\Lambda^{^{k\gamma}}\mathbb{G}(t)\interleave_{L^{p}}\|n_{0}\|_{L^1}\leq Ct^{-\big(\frac{3}{\alpha}+1\big)\big(1-\frac 1p\big)-\frac{k\gamma}{\alpha}}.
\end{equation}
Here the estimates for $\interleave\Lambda^{k\gamma}\mathbb{G}(t)\interleave_{L^{p}}$ could be obtained by Proposition \ref{prop:3.8}. Moreover, if $k>0$ , $\Lambda^{^{k\gamma}}$ is a pseudo-differential operator and  by Proposition \ref{prop:3.8}, one has that
\begin{equation}\label{eq:4.24}
\left\|\Lambda^{k\gamma}\left(\mathbb{G}(t)\circledast n_{0}\right)\right\|_{L^{p}}\leq
C\left\|\mathbb{G}(t)\circledast\Lambda^{k\gamma}n_{0}\right\|_{L^{p}}
\leq C\left\|\Lambda^{k\gamma}n_{0}\right\|_{L^p}.
\end{equation}
Thus, one could use \eqref{eq:4.24} to obtain the boundedness of this term. Combining the estimates \eqref{eq:4.23} and \eqref{eq:4.24}, we finish the proof of this lemma.
\end{proof}

Lemma \ref{lem:4.5} shows estimates of derivatives of the linear part in $n(x,y,z,t)$ as shown by \eqref{eq:4.8}. It is obvious that limited by the regularity of initial function $n_0$ (i.e. the condition $n_0\in W^{3, p}$), $k$ could not be larger than 6 in Lemma \ref{lem:4.5}. The reason that we consider the half order derivatives in the lemma, is a preparation for the nonlinear part. In fact, due to the singularities in the Green's function $\mathbb{G}$ and the structure of the nonlinearity $\nabla\cdot\left(n \mathbf{B}(n)\right)$, one could only consider half order derivatives step by step.

\begin{lemma}\label{lem:4.6}
Assuming $\gamma, k$ satisfy the condition \eqref{eq:4.22} and
$$
\left\|\Lambda^{j\gamma}n(\cdot, \cdot, \cdot, t)\right\|_{L^{p}}\leq C(1+t)^{-\big(\frac{3}{\alpha}+1\big)\big(1-\frac 1p\big)-\frac{j\gamma}{\alpha}}
$$
for $p\in[2,\infty)$ and $j=0,\cdots,k-1$, then one has that
$$
\|\Lambda^{k\gamma}n(\cdot, \cdot, \cdot, t)\|_{L^{p}}\leq C(1+t)^{-\big(\frac{3}{\alpha}+1\big)\big(1-\frac 1p\big)-\frac{k\gamma}{\alpha}}.
$$
\end{lemma}

\begin{proof}
Apply $\Lambda^{k\gamma}$ on the representation \eqref{eq:4.8} for $n$ to yield that
$$
\begin{aligned}
\Lambda^{k\gamma}n(x, y, z, t)=&\Lambda_{x, y, z}^{k\gamma}\int_{\mathbb{R}^3}\mathbb{G}(x, y, z, t;x', y', z')n_{0}(x', y', z')\, dx'\, dy'\, dz'\\
&+\int_{0}^{t}\int_{\mathbb{R}^3}\Lambda_{x, y, z}^{k\gamma}\mathbb{G}(x, y, z, t-s; x', y', z')\nabla\cdot\big(n \mathbf{B}(n)\big)(x', y', z', s)\, dx'\, dy'\, dz'\, ds,
\end{aligned}
$$
and thus combining Lemma \ref{lem:4.5} one has that
\begin{equation}\label{eq:4.25}
\begin{aligned}
&\left\|\Lambda^{k\gamma}n(\cdot, \cdot, \cdot, t)\right\|_{L^{p}}\\
\leq&\left\|\Lambda_{x, y, z}^{k\gamma}\mathbb{G}(t)\circledast n_{0}\right\|_{L^{p}}+\int_{0}^{{t}}\left\|\Lambda_{x, y, z}^{k\gamma}\Lambda_{x', y', z'}\mathbb{G}(t-s)\circledast\big(n \mathbf{B}(n)\big)(s)\right\|_{L^{p}}\, ds\\
\le & C(1+t)^{-\big(\frac{3}{\alpha}+1\big)\big(1-\frac 1p\big)-\frac {k\gamma}{\alpha}}+\int_{0}^{t/2}\left\|\Lambda_{x, y, z}^{k\gamma}\Lambda_{x', y', z'}\mathbb{G}(t-s)\circledast\big(n \mathbf{B}(n)\big)(s)\right\|_{L^{p}}\, ds\\
&+\int_{t/2}^{t}\left\|\Lambda_{x, y, z}^{k\gamma}\Lambda_{x', y', z'}\mathbb{G}(t-s)\circledast\big(n\mathbf{B}(n)\big)(s)\right\|_{L^{p}}\, ds.
\end{aligned}
\end{equation}
From Proposition \ref{prop:3.8}, Lemma \ref{lem:2.3} and Lemma \ref{lem:4.3}, one has that
\begin{equation*}
\begin{aligned}
&\int_{0}^{t/2}\left\|\Lambda_{x, y, z}^{k\gamma}\Lambda_{x', y', z'}\mathbb{G}(t-s)\circledast\big(n\mathbf{B}(n)\big)(s)\right\|_{L^{p}}\, ds\\
\leq&\int_{0}^{t/2}\left\|\Lambda_{x, y, z}^{k\gamma}\Lambda_{x', y', z'}\mathbb{G}(t-s)\right\|_{L^{p}}\|n(s)\|_{L^3}\|\mathbf{B}(n)(s)\|_{L^{3/2}}\, ds\\
\leq&C\int_{0}^{t/2}(t-s)^{-\frac{3}{\alpha}\big(1-\frac 1p\big)-\frac{k\gamma}{\alpha}}\big(1+A(t-s)\big)^{-1-\frac 1p}(1+s)^{-\big(\frac{3}{\alpha}+1\big)\frac{2}{3}}\, ds\\
\leq&C(1+t)^{-1-\frac 1p}t^{-\frac{3}{\alpha}\big(1-\frac 1p\big)-\frac{k\gamma}{\alpha}}.
\end{aligned}
\end{equation*}
If $t\geq 1$, we have
\begin{equation}\label{eq:4.27}
\int_{0}^{t/2}\|\Lambda_{x, y, z}^{k\gamma}\Lambda_{x', y', z'}\mathbb{G}(t-s)\circledast\big(n\mathbf{B}(n)\big)\|_{L^p}\, ds\leq C(1+t)^{-\big(\frac{3}{\alpha}+1\big)\big(1-\frac{1}{p}\big)-\frac{k\gamma}{\alpha}}.
\end{equation}
If $t<1$, we need the boundedness of this term.
\begin{equation}\label{eq:4.26}
\begin{aligned}
&\int_{0}^{t/2}\left\|\Lambda_{x, y, z}^{k\gamma}\Lambda_{x', y', z'}\mathbb{G}(t-s)\circledast\big(n\mathbf{B}(n)\big)(s)\right\|_{L^p}\, ds\\
\leq &\int_{0}^{t/2}\|\Lambda_{x, y, z}^{\gamma}\Lambda_{x', y', z'}\mathbb{G}(t-s)\|_{L^1} \|\Lambda^{(k-1)\gamma}\big(n \mathbf{B}(n)\big)(s)\|_{L^p}\, ds\\
\leq & \int_{0}^{t/2} (t-s)^{-\frac {1+\gamma}{\alpha}} \cdot \big(\|\Lambda^{(k-1)\gamma}n\|_{L^{2p}}\| \mathbf{B}(n)\|_{L^{2p}}+\|n\|_{L^{2p}}\|\Lambda^{(k-1)\gamma} \mathbf{B}(n)\|_{L^{2p}}\big)\, ds\\
\leq & C\int_{0}^{t/2}(t-s)^{-\frac{1+\gamma}{\alpha}}ds\leq Ct^{\frac{\alpha-1-\gamma}{\alpha}}.
\end{aligned}
\end{equation}
The estimate of \eqref{eq:4.27} and \eqref{eq:4.26} result in that
\begin{equation}\label{eq:4.28}
\int_{0}^{t/2}\left\|\Lambda_{x, y, z}^{k\gamma}\Lambda_{x', y', z'}\mathbb{G}(t-s)\circledast\big(n\mathbf{B}(n)\big)(s)\right\|_{L^{p}}\, ds\leq C(1+t)^{-\big(\frac{3}{\alpha}+1\big)\big(1-\frac 1p\big)-\frac{k\gamma}{\alpha}}
\end{equation}
since $\gamma<\alpha-1$ which ensures $\alpha-1-\gamma> 0$.

\vskip .05in

In the time interval $[t/2,t]$ near $t$,  Proposition \ref{prop:3.8}, Lemma \ref{lem:2.3} and Lemma \ref{lem:4.3}  yield that
\begin{equation}\label{eq:4.29}
\begin{aligned}
&\int^t_{t/2}\left\|\Lambda_{x, y, z}^{k\gamma}\Lambda_{x', y', z'}\mathbb{G}(t-s)\circledast\big(n\mathbf{B}(n)\big)(s)\right\|_{L^p}\, ds\\
\leq &\int^t_{t/2}\|\Lambda_{x, y, z}^{\gamma}\Lambda_{x', y', z'}\mathbb{G}(t-s)\|_{L^1} \|\Lambda^{(k-1)\gamma}\big(n\mathbf{B}(n)\big)(s)\|_{L^p}\, ds\\
\leq & \int^t_{t/2} (t-s)^{-\frac {1+\gamma}{\alpha}}  \cdot \big(\|\Lambda^{(k-1)\gamma}n\|_{L^{2p}}\| \mathbf{B}(n)\|_{L^{2p}}+\|n\|_{L^{2p}}\|\Lambda^{(k-1)\gamma} \mathbf{B}(n)\|_{L^{2p}}\big)\, ds\\
\leq & C\int^t_{t/2}(t-s)^{-\frac{1+\gamma}{\alpha}} (1+s)^{-\big(1+\frac{3}{\alpha}\big)\big(\frac{5}{3}-\frac{1}{p}\big)-\frac{(k-1)\gamma}{\alpha}}\, ds\\
\leq & C(1+t)^{-\big(\frac{3}{\alpha}+1\big)\big(1-\frac 1p\big)-\frac{k\gamma}{\alpha}}.
\end{aligned}
\end{equation}
The last inequality uses $\alpha-1-\gamma>0$ and the following estimates:
$$
\frac{\alpha-1-\gamma}{\alpha}-\Big(\frac{3}{\alpha}+1\Big)\Big(\frac{5}{3}-\frac{1}{p}\Big)-\frac{(k-1)\gamma}{\alpha}\leq -\Big(\frac{3}{\alpha}+1\Big)\Big(1-\frac 1p\Big)-\frac{k\gamma}{\alpha},
$$
We complete the proof by combining \eqref{eq:4.25}, \eqref{eq:4.28} and \eqref{eq:4.29}.
\end{proof}

\begin{proof}[The proof of Theorem\ref{the:1.1}]
Using the initial condition \eqref{2000} and the Theorem \ref{thm:4.1}, we know that there exists a $T_{*}>0$ such that the Cauchy problem (\ref{eq:1.1}) has a local classical solution on $\mathbb{R}^3\times (0, T_{*}]$. Then we use the a priori estimates (\ref{eq:4.2}) to extend the local classical solution to the global solution. As the classical solution exists on $(0, T_{*}]$, we know that $\|n(\cdot, \cdot, \cdot, t)\|_{L^p}\leq C$ exists on $(0, T_{*}]$. Therefore, there exists a $T_{1}\in (0, T_{*}]$ such that $(\ref{eq:4.2})$ holds for $T=T_{1}$. Set
$$
T^{*}=\sup\{T | (\ref{eq:4.2})\ \text {holds}\}.
$$
Then $T^*\geq T_{1}>0$. We want to show that $T^*=\infty$. If $T^*<\infty$, by Lemma \ref{lem:4.2}, we know that (\ref{eq:4.5}) holds on $(0, T^*]$. Then using the blow-up criterion (\ref{eq:4.1}), we know that there exists $T^{**}>T^*$ such that the solution exists on $\mathbb{R}^3\times (0, T^{**}]$ which contradicts with the definition of $T^*$. Then we have $T^*=\infty$. Using the blow-up criterion again, we know that the classical solution exists for any time. For the decay estimates of this solution, it is given in Lemma \ref{lem:4.6}.
\end{proof}

\vskip .15in

\noindent \textbf{Acknowledgement.} S. Deng is supported by National Key R\&D Program of China (No. 2022YFA1007300) and Shanghai Science and Technology Innovation Action Plan (No. 21JC1403600). B. Shi is supported by  the Natural Science Foundation of Jiangsu Province (No. BK20241432). W. Wang is supported by National Nature Science Foundation of China (No. 12271357) and Shanghai Science and Technology Innovation Action Plan (No. 21JC1403600). Y. Wang is supported by National Nature Science Foundation of China (No. 12271357), Shanghai Science and Technology Innovation Action Plan (No. 21JC1403600) and the German Research Foundation (No. CH 955/8-1).

\end{document}